\documentclass[11pt]{amsart}

\usepackage[a4paper,margin=1in]{geometry}
\usepackage[T1]{fontenc}
\usepackage{lmodern}
\usepackage{microtype}
\usepackage{amsmath,amssymb,mathtools}
\usepackage{enumitem}
\usepackage{xurl}
\usepackage{xcolor}
\usepackage{tikz}
\usetikzlibrary{arrows.meta,calc,positioning}
\usepackage[colorlinks=true,linkcolor=blue!45!black,citecolor=blue!45!black,urlcolor=blue!45!black]{hyperref}

\usepackage[all]{xy}
\usepackage{relsize}
\usepackage[bbgreekl]{mathbbol}

\DeclareSymbolFontAlphabet{\mathbb}{AMSb}
\DeclareSymbolFontAlphabet{\mathbbl}{bbold}

\numberwithin{equation}{section}

\theoremstyle{plain}
\newtheorem{thm}{Theorem}[section]
\newtheorem{prop}[thm]{Proposition}
\newtheorem{lem}[thm]{Lemma}
\newtheorem{cor}[thm]{Corollary}

\theoremstyle{definition}
\newtheorem{defn}[thm]{Definition}
\newtheorem{notation}[thm]{Notation}

\newtheorem{construction}[thm]{Construction}

\newtheorem{computation}[thm]{Computation}

\theoremstyle{remark}
\newtheorem{rem}[thm]{Remark}
\newtheorem{ex}[thm]{Example}
\newtheorem{claim}[thm]{Claim}

\newcommand{\Z}{\mathbb{Z}}
\newcommand{\tilf}{\widetilde{f}}
\newcommand{\tilg}{\widetilde{g}}
\newcommand{\tilb}{\widetilde{b}}

\newcommand{\tilalpha}{\widetilde{\alpha}}
\newcommand{\cj}{\mathrm{conj}}

\newcommand{\Fil}{\mathrm{Fil}}

\newcommand{\calN}{\mathcal{N}}

\newcommand{\Einfty}{\mathbf{E}_\infty}
\newcommand{\C}{\mathbb{C}}
\newcommand{\N}{\mathbb{N}}

\newcommand{\B}{\mathbb{B}}
\newcommand{\calO}{\mathcal{O}}

\newcommand{\cofib}{\mathrm{cofib}}
\newcommand{\fib}{\mathrm{fib}}
\newcommand{\Q}{\mathbb{Q}}
\newcommand{\T}{\mathbb{T}}

\newcommand{\F}{\mathbb{F}}

\newcommand{\Hom}{\mathrm{Hom}}

\newcommand{\can}{\mathrm{can}}

\newcommand{\CycSp}{\mathrm{CycSp}}

\newcommand{\QSyn}{\mathcal{Q}\mathrm{Syn}}
\newcommand{\gr}{\mathrm{gr}}
\newcommand{\CAlg}{\mathrm{CAlg}}

\newcommand{\Ext}{\mathrm{Ext}}

\newcommand{\perf}{\mathrm{perf}}

\newcommand{\THH}{\mathrm{THH}}
\newcommand{\HH}{\mathrm{HH}}

\newcommand{\TC}{\mathrm{TC}}
\newcommand{\TP}{\mathrm{TP}}
\newcommand{\Sphere}{\mathbb{S}}

\newcommand{\Prism}{\mathbbl{\Delta}}

\newcommand{\Fun}{\mathrm{Fun}}

\newcommand{\ima}{\mathrm{im}}

\newcommand{\coker}{\mathrm{coker}}

\newcommand{\rmD}{\mathrm{D}}

\title[A Note on Topological Hochschild Homology Relative to \texorpdfstring{$\Sphere_{W(k)}[x_0,x_1,\ldots,x_n]$}{}]
{A Note on Topological Hochschild Homology Relative to \texorpdfstring{$\Sphere_{W(k)}[x_0,x_1,\ldots,x_n]$}{}}

\begin{document}

\author{Jingbang Guo} 
\thanks{The author would thank the support from Guozhen Wang and the funding NSFC-12226002.}

\begin{abstract}
   We explain the relation between the relative topological Hochschild homology $\THH(R/\Sphere_{W(k)}[x_0,\ldots,x_n])$ and the Nygaard-completed Frobenius-twisted relative prismatic cohomology $\widehat{\Prism}^{(1)}_{R/W(k)[x_0,\ldots,x_n]^\wedge}$, where $W(k)[x_0,x_1,\ldots,x_n]\rightarrow R$ is relatively quasiregular semiperfectoid. As an application, for $R=\Z_p[x]/(px)$, we compute $\pi_*\THH(R)^\wedge_p$ by descent along $\THH(R)^\wedge_p\rightarrow \THH(R/\Sphere_p[z,x])$, where $R$ is regarded as an $\Einfty$-$\Sphere_p[z,x]$-algebra via $\Sphere_p[z,x]\xrightarrow{z\mapsto p,x\mapsto x}\Z_p[x]/(px)$.
\end{abstract}

\maketitle

{\setcounter{tocdepth}{1}\tableofcontents}

%%%%%%%%%%%%%%%%%%%%%%%%%%%%%%%%%%%%%%%%%%%%%%%%%%%%%%%%%%%%%%%%%%%%%%%%%%%%%%%%%%%%%%%
\section{Introduction}\label{section:introduction}

Let $p$ be a prime number. Let $k$ be a perfect ring of characteristic $p$, and let $\Sphere_{W(k)}$ be the corresponding ring of spherical Witt vectors with coefficients in $k$. Set 
    $$
    \B_n:=\Sphere_{W(k)}[x_0,x_1,\ldots,x_n]=\Sphere_{W(k)}\otimes (\Sigma^{\infty}_+\N)^{\otimes (n+1)}.
    $$    
Also, let $B_n=W(k)[x_0,x_1,\ldots,x_n]$ be equipped with the $\delta$-structure determined by the canonical $\delta$-structure on $W(k)$ and by $\delta(x_i)=0$ for $0\leq i\leq n$. 

In this note, we explain the relation between the relative topological Hochschild homology $\THH(R/\B_n)$ and the Nygaard-completed Frobenius-twisted relative prismatic cohomology $\widehat{\Prism}^{(1)}_{R/B_n^\wedge}$, where the map $W(k)[x_0,x_1,\ldots,x_n]\rightarrow R$ is relatively quasiregular semiperfectoid (in the sense of Definition \ref{defn:qrsp-relative}, which is a restricted version of \cite[Definition 9.2]{AKN23} or \cite[Definition 15.1]{BS22}).

As an application, for $R=\Z_p[x]/(px)$, we compute $\pi_*(\THH(R)^\wedge_p)$ by descent along the map $\THH(R)^\wedge_p\rightarrow \THH(R/\Sphere_p[z,x])$, where $R$ is regarded as an $\Sphere_p[z,x]$-algebra via the map $\Z_p[z,x]\xrightarrow{z\mapsto p,x\mapsto x}\Z_p[x]/(px)$.

\subsection{Main Results}\label{subsection:main-results}

\begin{thm}[Theorem \ref{thm:THH-for-qrsp-relative-polynomial}, Corollary \ref{cor:THH-delta-power-Nygaard}]\label{thm:intro-THH-for-qrsp-rel}
    Let $R$ be a commutative ring with a relatively quasiregular semiperfectoid map $B_n\rightarrow R$; more precisely, the following conditions are satisfied:
    \begin{itemize}
        \item[(i)] The ring $R$ has bounded $p$-power torsion and is $p$-complete. 

        \item[(ii)] The ring homomorphism $B_n\rightarrow  R$ is surjective. 

        \item[(iii)] The ideal $\ker(B_n\rightarrow R)$ contains a distinguished ideal $I$ of $B_n$ (see Definition \ref{defn:distinguished-ideals}), and $L\Omega^1_{R/(B_n/I)}\in \rmD(R)$ has $p$-complete Tor-amplitude concentrated in homological degree $1$.
    \end{itemize} 
    Now let $B_n^\wedge$ denote the $(p,I)$-completion of $B_n$, and regard $(B_n^\wedge, IB_n^\wedge)$ as a prism. Then the following statements hold:
    \begin{itemize}
        \item[(1)] The homotopy groups of $\THH(R/\B_n)$ are concentrated in even degrees, and there is a canonical isomorphism of graded rings 
        $$
        \pi_{2*} \THH(R/\B_n)\simeq \gr^*_\calN\widehat{\Prism}^{(1)}_{R/B_n^\wedge}.
        $$

        \item[(2)] The homotopy groups of $\TC^-(R/\B_n)$ are concentrated in even degrees, and there is a canonical isomorphism of graded rings 
        $$
        \pi_{2*}\TC^-(R/\B_n)\simeq \bigoplus_{*\in \Z\ \text{graded}}\calN^{\geq *}\widehat{\Prism}^{(1)}_{R/B_n^\wedge}\{*\}.
        $$
        %\item[(4)] Under the identifications in (2) and (3), the canonical map $\can:\TC^-(R/\B_n)\rightarrow \TP(R/\B_n)$ induces on homotopy groups the canonical inclusions for Nygaard filtrations.

        %\item[(5)] Under the identifications (2), and (3), the cyclotomic Frobenius $\varphi_p^{h\T}:\TC^-(R/\B_n)\rightarrow \TP(R/\B_n)$ induces on homotopy groups the following map 
        %$$
        %\bigoplus_{*\in \Z\ \text{graded}}\calN^{\geq *}\widehat{\Prism}^{(1)}_{R/B_n^\wedge}\{*\}\xrightarrow{c\circ\phi_{/B_n^\wedge}}\bigoplus_{*\in \Z\ \text{graded}}\widehat{\Prism}^{(1)}_{R/B_n^\wedge}\{*\},
        %$$
        %where $c$ and $\phi_{/B_n^\wedge}$ are maps in Notation \ref{notation:relative-prismatic-cohomology}, (e) and (f), respectively. 

        %Consider the following commutative diagram:
        %$$
        %\xymatrix@R=50pt@C=50pt{
        %\TC^-(R/\B_n) \ar[r]^{\varphi_p^{h\T}} \ar[d] & %\TP(R/\B_n) \ar[d] \\
        %\THH(R/\B_n) \ar[r]^{\varphi_p} & \THH(R/\B_n)^{tC_p}.
        %}
        %$$
        %With the identifications in (1), (2), and (3), by taking the homotopy groups of the above commutative diagram we obtain the following commutative diagram of graded rings:
        %$$
        %\xymatrix@R=50pt@C=50pt{
        %\bigoplus_{*\in \Z\ \text{graded}}\calN^{\geq %*}\widehat{\Prism}^{(1)}_{R/B_n^\wedge}\{*\} 
        %\ar[r]^{c\circ \phi_{/B_n^\wedge}} \ar[d] & \bigoplus_{*\in \Z\ %\text{graded}}\widehat{\Prism}^{(1)}_{R/B_n^\wedge}\{*\} \ar[d] \\
        %\gr^*_\calN\widehat{\Prism}^{(1)}_{R/B_n^\wedge} \ar[r]^-{c\circ %\gr^*\phi_{/B_n^\wedge}} & \bigoplus_{*\in \Z\ \text{graded}}%(\widehat{\Prism}^{(1)}_{R/B_n^\wedge}\{*\}/\phi(I)).
        %}
        %$$
        \item[(3)]  More specifically, if $I=(d)\subset B_n$ is a distinguished principal ideal and $R=B_{n}/(I,f_1,\ldots,f_r)$, where $f_1,\ldots,f_r\in B_n$ form a Koszul regular sequence in $\overline{B}_n:=B_{n}/I$. Then there is a canonical isomorphism of graded rings 
        $$
        \pi_{2*}\TC^-(R/\B_n)\simeq \bigoplus_{*\in \Z\ \text{graded}}\calN^{\geq *}B_n^\wedge\{\frac{\phi(f_1)}{\phi(d)},\ldots,\frac{\phi(f_r)}{\phi(d)}\}^\wedge_{(p,\calN)}\{*\}.
        $$
        
        \item[(4)] In the situation of (3), we further assume that either of the following conditions is satisfied: 
    \begin{itemize}
        \item[(i)] $R$ is $p$-torsion-free and $\phi^k(f_i)\in (d,f_1,\ldots,f_r)$ for $k\geq 0$ and $1\leq i\leq r$.

        \item[(ii)] $\phi(f_i)=g_{i,1}-g_{i,0}$ with $\delta(g_{i,0})=\delta(g_{i,1})=0$ for $1\leq i\leq r$.
    \end{itemize} 
    Then there is an isomorphism of graded rings with divided power structures
    $$
    \pi_{2*}(\THH(R/\B_n))\simeq \overline{B}_n[u]\otimes_{\overline{B}_n}R\langle t_{1},\ldots,t_r\rangle,
    $$
    where $u$ corresponds to the class of $d\in \calN^{\geq 1}B^\wedge_n\{\frac{\phi(f_1)}{\phi(d)},\ldots,\frac{\phi(f_r)}{\phi(d)}\}^\wedge_{(p,\calN)}$, and, up to multiplication by a unit, $t_i^{[p^k]}$ corresponds to the class of $f_i^{(k)}\in \calN^{\geq p^k}B^\wedge_n\{\frac{\phi(f_1)}{\phi(d)},\ldots,\frac{\phi(f_r)}{\phi(d)}\}^\wedge_{(p,\calN)}$ (see Construction \ref{construction:generators-for-envelopes}).
    \end{itemize}

\end{thm}

\begin{rem}\label{rem:intro-THH-qrsp-rel}
The relative topological Hochschild homology $\THH(-/\Sphere_{W(k)}[x_0,\ldots,x_n])$ has been studied in \cite[\S 11]{BMS19}, \cite{KN22} and \cite{LW22}, where one can find special but essential cases of Theorem \ref{thm:intro-THH-for-qrsp-rel}; see in particular \cite[Proposition 11.10]{BMS19}, \cite[Proposition 3.7, Theorem 9.1]{KN22}, and \cite[\S 3]{LW22}. The key idea of Theorem \ref{thm:intro-THH-for-qrsp-rel} has already appeared in \cite[Proposition 11.10]{BMS19}. A more advanced version of Theorem \ref{thm:intro-THH-for-qrsp-rel} allowing bases other than spectra of the form $\B_n$ and relating $\TC^-$ to prismatic cohomology relative to $\delta$-rings (in the sense of \cite{AKN23}) has been announced in \cite{Krause23}.
\end{rem}

We also compute $\THH(\Z_p[x]/(px))^\wedge_p$ by descent along the map $\THH(\Z_p[x]/(px))^\wedge_p\rightarrow \THH((\Z_p[x]/(px))/\Sphere_p[z,x])$, where $\Z_p[x]/(px)$ is regarded as an $\Sphere_p[z,x]$-algebra via the map $\Sphere_p[z,x]\xrightarrow{z\mapsto p,x\mapsto x}\Z_p[x]/(px)$.

\begin{thm}[Theorem \ref{thm:computation}]\label{thm:intro-computation}
    Let $R=\Z_p[x]/(px)$. There is a multiplicative second-quadrant homology-type spectral sequence $\{E^n_{i,j}\}$ converging to $\pi_*(\THH(R)^\wedge_p)$, such that 
    \begin{itemize}
        \item[(a)] This spectral sequence collapses at the $E^2$-page.

        \item[(b)] The $E^2$-page $E^2_{-*,\star}$ can be computed by the following complex of graded $R$-modules (here $*$ corresponds to the cohomological degree, and $\star$ corresponds to the grading)
        $$
        R\langle t_{zx}\rangle[u]\xrightarrow{(D'_zdz,D'_x dx)} R\langle t_{zx}\rangle[u]dz\oplus R\langle t_{zx}\rangle[u]dx\xrightarrow{D'_xdx\oplus D'_zdz} R\langle t_{zx}\rangle[u](dz\wedge dx),
        $$
        where $R\langle t_{zx}\rangle[u]$ denotes the polynomial ring in $u$ with divided powers in $t_{zx}$, and $|u|=2$, $|t_{zx}^{[k]}|=2k$, and $|dz|=|dx|=2$, $|dz\wedge dx|=4$, and $D'_z: R\langle t_{zx}\rangle[u]\rightarrow  R\langle t_{zx}\rangle[u]$ is the $R$-linear map sending $t_{zx}^{[k]}u^n$ to $n t_{zx}^{[k]}u^{n-1}+x t_{zx}^{[k-1]}u^n$, and $D'_x: R\langle t_{zx}\rangle[u]\rightarrow  R\langle t_{zx}\rangle[u]$ is the $R[u]$-linear map sending $t_{zx}^{[k]}$ to $pt_{zx}^{[k-1]}$.
    \end{itemize}
    
    %Now, for $l\geq 1$, we set 
    %\begin{align*}
    %    \epsilon_{l} = \sum_{n+k=l,p\lfloor\frac{l}{p}\rfloor\leq n\leq l}(-1)^{n-p\lfloor\frac{l}{p}\rfloor}\frac{x^{n-p\lfloor\frac{l}{p}\rfloor}}{n(n-1)\cdots(p\lfloor\frac{l}{p}\rfloor+1)} u^n\cdot t_{zx}^{[k]};
    %\end{align*}
    %and for $0\leq l'< l$ such that $l'+1$ is not divisible by %$p$, we set 
    %$$
    %\beta_{l,l'}=u^{l'}\cdot t_{zx}^{[l-l']}dz+\sum_{n=l'+1}^{\min\{p(\lfloor \frac{l'}{p}\rfloor+1)-1,l\}}\frac{(-x)^{n-l'}p}{n(n-1)\cdots(l'+1)}u^n\cdot t_{zx}^{[l-n]}dx;
    %$$
    %also, for $1\leq i\leq \lfloor\frac{l}{p}\rfloor$, we set 
    %$$
    %\gamma_{l,i}=iu^{pi-1}\cdot t_{zx}^{[l-pi+1]}dz+\sum_{n=pi}^{\min\{p(i+1)-1,l\}}\frac{(-x)^{n-pi}}{n(n-1)\cdots(pi+1)}u^n\cdot t_{zx}^{[l-n]}dx,
    %$$
    %and finally,
    %$$
    %\gamma_{l,0}=\sum_{n=0}^{\min\{p-1,l\}}\frac{(-x)^{n-pi}}{n(n-1)\cdots(pi+1)}u^n\cdot t_{zx}^{[l-n]}dx.
    %$$
    Moreover, the following statements hold:
    \begin{itemize}
       \item[(1)] We have $E^2_{i,j}=0$ if $i\neq 0,-1,-2$ or $j$ is odd.
    
        \item[(2)] We have $E^2_{0,0}\simeq R$. For $l\geq 1$, we have 
        $$
        E^2_{0,2l}\simeq \F_p[x].
        $$

        \item[(3)] We have $E^2_{-1,0}= 0$, $E^2_{-1,2}\simeq \F_p[x]$. Let $l\geq 1$ be an integer. Then the following statements hold:
        \begin{itemize}
            \item[(i)] If $p\nmid l+1$, then $E^2_{-1,2l+2}$ fits into a non-split exact sequence 
            $$
            0\rightarrow \oplus_{1\leq i\leq \lfloor\frac{l}{p}\rfloor}\F_p[x]/(x^{p+1})\rightarrow E^2_{-1,2l}\rightarrow  \oplus_{0\leq i\leq \lfloor\frac{l}{p}\rfloor}\F_p[x]\rightarrow 0.
            $$

            \item[(ii)] If $p\mid (l+1)$, then $E^2_{-1,2l+2}$ fits into a non-split short exact sequence 
            $$
            0\rightarrow R/(l+1,x^{p+1})\oplus(\oplus_{1\leq i\leq \lfloor\frac{l}{p}\rfloor}\F_p[x]/(x^{p+1}))\rightarrow E^2_{-1,2l}\rightarrow \oplus_{0\leq i\leq \lfloor\frac{l}{p}\rfloor}\F_p[x]\rightarrow 0.
            $$
        \end{itemize}

        \item[(4)] We have $E^2_{-2,0}=E^2_{-2,2}=E^2_{-2,4}=0$. Let $l\geq 1$ be an integer. Then 
        $$
        E^2_{-2,2l+4}\simeq \oplus_{1\leq i\leq \lfloor\frac{l+1}{p}\rfloor}\F_p[x]/(x^p).
        $$
    \end{itemize}
\end{thm}

\begin{rem}\label{rem:intro-computation}
    The extension problem for the spectral sequence in Theorem \ref{thm:intro-computation} remains unsolved, and the multiplicative structure has not been specified. However, in a forthcoming work with Wei Yang \cite{GY26}, we will resolve the extension problem and specify the multiplicative structure by choosing a suitable differential graded algebra model of $\THH(\Z_p[x]/(px))^\wedge_p$. 
\end{rem}

\subsection{Main Idea}\label{subsection:main-idea}

By considering base change along 
$$
\Sphere_{W(k)}[x_0,x_1,\ldots,x_n]\rightarrow \Sphere_{W(k)}[x_0^{1/p^\infty},x_1^{1/p^\infty},\ldots,x_n^{1/p^\infty}],
$$ 
and noting that the relative prismatic cohomology data satisfies base change (see, for example \cite[Theorem 1.2 (7)]{AKN23}), we prove Theorem \ref{thm:intro-THH-for-qrsp-rel} by descent from the following result:
\begin{thm}[{\cite[\S 7.2]{BMS19}}, {\cite[\S 13]{BS22}}]\label{thm:intro-THH-for-qrsp}
    Let $S$ be a quasiregular semiperfectoid ring equipped with a ring homomorphism $R\rightarrow S$, where $R$ is a perfectoid ring with the Fontaine map $\theta:W(R^\flat)\rightarrow R$. Regard $(W(R^\flat),\ker\theta)$ as a prism. %Let $\calN^{\geq \star}\widehat{\Prism}^{(1)}_{S/W(R^\flat)}\{i\}$, $i\in \Z$, denote the Nygaard filtration on the Nygaard-completed, Frobenius twisted, $i$-th Breuil-Kisin twisted prismatic cohomology $\widehat{\Prism}_{S/W(R^\flat)}^{(1)}\{i\}$, and let $\gr_\calN^*\widehat{\Prism}_{S/W(R^\flat)}^{(1)}\{i\}$ denote the associated graded. Also, we regard $S$ naturally as an $\Sphere_{W(R^\flat)}$-algebra. 
    Then the following statements hold:
    \begin{itemize}
        \item[(1)] The homotopy groups of $\THH(S/\Sphere_{W(R^\flat)})$ are concentrated in even degrees, and there is a canonical isomorphism of graded rings 
        $$
        \pi_{2*} \THH(S/\Sphere_{W(R^\flat)})\simeq \gr^*_\calN\widehat{\Prism}^{(1)}_{S/W(R^\flat)}.
        $$

        \item[(2)] The homotopy groups of $\TC^-(S/\Sphere_{W(R^\flat)})$ are concentrated in even degrees, and there is a canonical isomorphism of graded rings 
        $$
        \pi_{2*}\TC^-(S/\Sphere_{W(R^\flat)})\simeq \bigoplus_{*\in \Z\ \text{graded}}\calN^{\geq *}\widehat{\Prism}^{(1)}_{S/W(R^\flat)}\{*\}.
        $$

    \end{itemize}
\end{thm}

\begin{rem}\label{rem:intro-LW}
    In the situation of Theorem \ref{thm:intro-THH-for-qrsp-rel} (3) with the further assumption that $R$ is $p$-torsion-free, one can also analyze $\TC^-(R/\B_n)$ directly as in \cite[\S 3]{LW22}.
\end{rem}

In \cite[\S 4 and \S 5]{LW22}, $\THH(\calO_K)^\wedge_p$ (where $\calO_K$ is the ring of integers of a number field $K$) is computed by descent along $\THH(\calO_K)^\wedge_p\rightarrow \THH(\calO_K/\Sphere_{W(k)}[z])$ (where $k$ is the residue field of $K$ and $z$ is sent to a chosen uniformizer $\varpi\in \calO_K$). We follow the same strategy to compute $\THH(R)^\wedge_p$, where $R=\Z_p[x]/(px)$. The spectral sequence in Theorem \ref{thm:intro-computation} converging to $\pi_*\THH(R)^\wedge_p$ is precisely induced by the coskeleton filtration of $\THH(R/\Sphere_p[z,x]^{\otimes_{\Sphere_p}(\bullet+1)})$: by \cite[Proposition 2.14]{MNN17}, we have an equivalence    
    $$
    \THH(R)^\wedge_p\rightarrow \varprojlim_{\bullet\in \Delta}\THH(R/\Sphere_p[z,x]^{\otimes_{\Sphere_p}(\bullet+1)}).
    $$
The $E^2$-page of this spectral sequence can be computed as the cohomology of the following Hopf algebroid (see Lemma \ref{lem:Hopf-algebroid})
        \begin{align*}
        \ & (\pi_*\THH(R/\Sphere_p[z,x]),\pi_*\THH(R/\Sphere_p[z_0,x_0,z_1,x_1]))\\
        & \simeq (R\langle t_{zx}\rangle[u],R\langle t_{z_1-z_0},t_{x_1-x_0},t_{z_0x_0}\rangle[u_0]).
        \end{align*}
Choosing a suitable resolution by relative injectives, we obtain the results in Theorem \ref{thm:intro-computation}. Finally, we note that we are essentially computing the spectral sequence associated with the motivic filtration for $\THH(R)^\wedge_p$ following the idea of \cite[Proposition 7.13]{BMS19}; in other words, we are actually computing 
$$
\gr^*_\calN \widehat{\Prism}_{R}\simeq \varprojlim_{\bullet\in \Delta}\gr^*_\calN\widehat{\Prism}^{(1)}_{R/\Z_p[z,x]^{\otimes_{\Z_p}(\bullet+1)}}.
$$
More precisely, we can consider the filtration $\varprojlim_{\bullet\in \Delta}\tau_{\geq \star}\THH(R/\Sphere_p[z,x]^{\otimes_{\Sphere_p}(\bullet+1)})$ on $\THH(R)^\wedge_p$ and appeal to \cite[Proposition 7.13]{BMS19}, \cite[Theorem 1.2 (6)]{AKN23}, and \cite[Theorem 2.3.4]{ABBK26}. 

\begin{rem}\label{rem:intro-motivic-pullback}
    In \cite{BL25}, the algebraic $\mathrm{K}$-theory of $\Z[x]/(px)$ is studied using the theory of motivic pullback squares developed in \cite{LT23}. A similar method can be applied to $\THH(\Z[x]/(px))$. 
\end{rem}

\begin{rem}\label{rem:intro-KN}
    One can also compute $\pi_*\THH(\Z_p[x]/(px))^\wedge_p$ by applying \cite[Proposition 7.1]{KN22} to $\THH((\Z_p[x]/(px))/\Sphere_p[z,x])$: there is a multiplicative convergent spectral sequence 
    $$
    \pi_i(\THH((\Z_p[x]/(px))/\Sphere_p[z,x])\widehat{\otimes}_{\Z_p[z,x]}\pi_j(\HH(\Z_p[z,x]/\Z_p)^\wedge_p))\Rightarrow \pi_{i+j}\THH((\Z_p[x]/(px))^\wedge_p.
    $$
\end{rem}

\subsection*{Acknowledgements}
The author is very grateful to Guozhen Wang for suggesting this investigation. The author is especially grateful to Wei Yang for enlightening conversations on various methods to compute $\THH(\Z_p[x]/(px))^\wedge_p$. The author also thanks Hui Gao and Yupeng Wang for valuable discussions.

%%%%%%%%%%%%%%%%%%%%%%%%%%%%%%%%%%%%%%%%%%%%%%%%%%%%%%%%%%%%%%%%%%%%%%%%%%%%%%%%%%%%%%%
\section{A Reminder on Relative \texorpdfstring{$\THH$}{THH}}\label{section:reminder-on-relative-THH}

In this section, we briefly review the basic definitions for relative topological Hochschild homology. Our main references are \cite[\S 11]{BMS19}, \cite[\S 2]{LW22}, and \cite{NS18}. We consider only the $\THH$ of $\Einfty$-rings. 

\begin{defn}[Relative Topological Hochschild Homology]\label{defn:relative-THH}
    Let $E$ be an $\Einfty$-ring spectrum. The {\it topological Hochschild homology relative to $E$} is the following functor 
    \begin{align*}
    \THH(-/E):  \CAlg_E&\rightarrow \CAlg_E^{B\T}  \\
                A  & \mapsto \THH(A/E):=A^{\otimes_{E}\T}=A\otimes_{A\otimes_E A} A, 
    \end{align*}
    where $\CAlg_E$ is the $\infty$-category of $\Einfty$-$E$-algebras, and $\T=\{z\in \C \mid |z|=1\}$ is the circle group, and $\CAlg_E^{B\T}=\Fun(B\T,\CAlg_E)$ is the $\infty$-category of $\T$-equivariant $\Einfty$-$E$-algebras.

    We define the {\it topological negative cyclic homology relative to $E$} to be the following functor:
    $$
    \TC^-(-/E):=\THH(-/E)^{h\T}:\CAlg_E\rightarrow \CAlg_E,
    $$
    where $(-)^{h\T}$ denotes the homotopy fixed points with respect to the $\T$-action. We also define the {\it topological periodic cyclic homology relative to $E$} as the following functor:
    $$
    \TP(-/E):=\THH(-/E)^{t\T}:\CAlg_E\rightarrow \CAlg_E,
    $$
    where $(-)^{t\T}$ denotes the Tate construction with respect to the $\T$-action.

    When $E=\Sphere$, the sphere spectrum, we simply write $\THH(-)$, $\TC^-(-)$, and $\TP(-)$.
\end{defn}

\begin{rem}\label{rem:THH}
    By definition, $\THH(-/E):\CAlg_E\rightarrow \CAlg_E^{B\T}$ is left adjoint to the forgetful functor $\CAlg_E^{B\T}\rightarrow \CAlg_E$; in other words, for an $\Einfty$-$E$-algebra $A$, there is a non-equivariant $\Einfty$-$E$-algebra map $A\rightarrow \THH(A/E)$ that is universal among non-equivariant $\Einfty$-$E$-algebra maps from $A$ to $\Einfty$-$E$-algebras with $\T$-actions. 
    
    Note that $\THH(-/-)$ is multiplicative in the sense that 
    $$
    \THH(A_1/E_1)\otimes_{\THH(A_2/E_2)}\THH(A_3/E_3)\simeq \THH(A_1\otimes_{A_2}A_3/E_1\otimes_{E_2}E_3).
    $$
    
    Also note that for an $\Einfty$-$E$-algebra $A$, there is the $\T$-equivariant augmentation map 
    $$
    \THH(A/E)\rightarrow \THH(A/A)=A,
    $$
    where the target is $A$ equipped with the trivial $\T$-action.
\end{rem}

Let $\Sphere$ denote the sphere spectrum, and let $A$ be an $\Einfty$-ring. It is known that $\THH(A):=\THH(A/\Sphere)$ admits not just a $\T$-action, but also the structure of a cyclotomic spectrum. 

\begin{defn}[Cyclotomic Spectra, {\cite[Definition II.1.1]{NS18}}]\label{defn:cyclotomic-spectra}
    A {\it cyclotomic spectrum} is a spectrum $X$ with a $\T$-action and $\T$-equivariant maps $\varphi_p:X\rightarrow X^{tC_p}$ for every prime $p$; we refer to $\varphi_p$ as the {\it cyclotomic Frobenius of $X$ at a prime $p$}. Let $\CycSp$ denote the {\it $\infty$-category of cyclotomic spectra}. See \cite[Definition II.1.6]{NS18} for more details. 

    For a fixed prime $p$, a {\it $p$-cyclotomic} spectrum is a spectrum $X$ with a $\T$-action and a $\T$-equivariant map $\varphi_p:X\rightarrow X^{tC_p}$; we refer to $\varphi_p$ as the {\it cyclotomic Frobenius of $X$}.
    Let $\CycSp_p$ denote the {\it $\infty$-category of $p$-cyclotomic spectra}.

    Note that in \cite[Definition II.1.1]{NS18}, a $p$-cyclotomic spectrum is defined to be a spectrum $X$ with a $C_{p^\infty}$-action and a $C_{p^\infty}$-equivariant map $\varphi_p:X\rightarrow X^{tC_p}$. However, in this note the spectrum $X$ in consideration will come from $\THH$ and be $p$-complete and bounded below, and by \cite[Remark II.1.3]{NS18}, there is no essential difference between $\T$ and $C_{p^\infty}$ for such an $X$. 
\end{defn}

\begin{prop}\label{prop:THH-as-cyclotomic-spectrum}
    For an $\Einfty$-ring $A$, $\THH(A)=\THH(A/\Sphere)$ is a cyclotomic spectrum.
\end{prop}

\begin{proof}
    See, for example, \cite[Corollary IV.2.3]{NS18}.
\end{proof}

For a general $\Einfty$-ring $E$, $\THH(-/E)$ may admit no nontrivial cyclotomic structure: for example, the proof of \cite[Corollary IV.2.3]{NS18} cannot be adapted to $\HH(-):=\THH(-/\Z)$ since there is no nontrivial Tate diagonal for $\rmD(\Z)$; see \cite[Theorem III.1.10]{NS18}. Nevertheless, we have some conditions on $E$ guaranteeing that $\THH(-/E)$ can be enhanced to a functor $\CAlg_E\rightarrow \CycSp_E$.

\begin{defn}[Cyclotomic bases, {\cite[Lemma 2.3]{LW22}}]\label{defn:cyclotomic-base}
    A {\it cyclotomic base} is an $\Einfty$-ring object $E$ in cyclotomic spectra such that its underlying $\T$-action is trivial and it is equipped with a lift of the augmentation map $\THH(E)\rightarrow E$ to a map of $\Einfty$-rings in cyclotomic spectra.

    Let $E_1$ and $E_2$ be cyclotomic bases. A {\it morphism from $E_1$ to $E_2$} is a commutative diagram of $\Einfty$-rings in cyclotomic spectra:
        $$
        \xymatrix@R=50pt@C=50pt{
        \THH(E_1) \ar[r] \ar[d] & E_1\ar[d] \\
        \THH(E_2) \ar[r]        & E_2.
        }
        $$
\end{defn}

%By \cite[Proposition 11.3]{BMS19}, $\Sphere[z]:=\Sigma^\infty_+\N$ admits an $\Einfty$-cyclotomic structure with trivial $\T$-action, and also the augmentation $\THH(\Sphere[z])\rightarrow \Sphere[z]$ admits a lift to a map of $\Einfty$-algebra in cyclotomic spectra. In this article, we would consider a cyclotomic base of the form $\Sphere_{W(k)}[x_0,x_1,\ldots,x_n]$. We fix a prime number $p$.

\begin{lem}\label{lem:spherical-perfect}
    Let $p$ be a prime number, and let $E$ be a $p$-complete connective $\Einfty$-ring such that $E\otimes_\Sphere \F_p$ is a perfect ring of characteristic $p$ which is concentrated in degree $0$. Then the augmentation map $\THH(E)\rightarrow E$ induces an equivalence $\THH(E)^\wedge_p\xrightarrow{\sim}E$, and $E$ is naturally equipped with the structure of a cyclotomic base.    
\end{lem}

\begin{proof}
    It suffices to check that $\THH(E)\otimes_\Sphere \Sphere/p\rightarrow E\otimes_{\Sphere}\Sphere/p$ is an equivalence, and since everything is connective, it suffices to check that $\THH(E)\otimes_{\Sphere}\F_p\simeq \THH(E\otimes_\Sphere \F_p/\F_p)\rightarrow E\otimes_\Sphere \F_p$ is an equivalence. By assumption, $E\otimes_{\Sphere}\F_p$ is a perfect ring of characteristic $p$ concentrated in degree $0$, and it suffices to appeal to the HKR filtration (see, for example, the last paragraph of \cite[\S 2.2]{BMS19}) on $\THH(E\otimes_\Sphere \F_p/\F_p)$ with the observation that $L\Omega^1_{E\otimes_\Sphere \F_p/\F_p}\simeq 0$.
\end{proof}

\begin{cor}\label{cor:spherical-perfect}
    Let $p$ be a prime number, and let $E$ be a $p$-complete connective $\Einfty$-ring such that $E\otimes_\Sphere \F_p$ is a perfect ring of characteristic $p$ which is concentrated in degree $0$, and let $R$ be an $\Einfty$-$E$-algebra. Then the canonical map 
    $$
    \THH(R)^\wedge_p\rightarrow \THH(R/E)^\wedge_p
    $$
    is an equivalence. 
\end{cor}

\begin{proof}
    Note that 
    $$
    \THH(R/E)\simeq \THH(R)\otimes_{\THH(E)}E,
    $$
    and the statement follows immediately from Lemma \ref{lem:spherical-perfect}.
\end{proof}

\begin{ex}\label{ex:spherical-Witt-vectors}
    Let $k$ be a perfect ring of characteristic $p$, and let $\Sphere_{W(k)}$ be the $\Einfty$-ring of spherical Witt vectors constructed in \cite[Example 5.2.7]{Lurie_EllipticII18}. By loc. cit., $\Sphere_{W(k)}$ satisfies the hypothesis of Lemma \ref{lem:spherical-perfect}. Therefore, the augmentation map $\THH(\Sphere_{W(k)})\rightarrow \Sphere_{W(k)}$ induces an equivalence $\THH(\Sphere_{W(k)})^\wedge_p\xrightarrow{\sim}\Sphere_{W(k)}$, and $\Sphere_{W(k)}$ is naturally equipped with the structure of a cyclotomic base. 
\end{ex}

\begin{ex}\label{ex:polynomial-base}
    Let $k$ be a perfect ring of characteristic $p$, and let $\Sphere_{W(k)}$ be the corresponding ring of spherical Witt vectors with coefficients in $k$. Set 
    $$
    \B_n:=\Sphere_{W(k)}[x_0,x_1,\ldots,x_n]=\Sphere_{W(k)}\otimes (\Sigma^{\infty}_+\N)^{\otimes (n+1)}.
    $$
    By \cite[Proposition 11.3]{BMS19}, the spectrum $\B_n$ admits an $\Einfty$-cyclotomic structure with the trivial $\T$-action, and, moreover, the augmentation $\THH(\B_n)\rightarrow \B_n$ admits a lift to a map of $\Einfty$-algebras in cyclotomic spectra; in other words, $\B_n$ admits the structure of a cyclotomic base.

    %Also note that, by \cite[IV.1.2]{NS18} and \cite[Proposition 11.3]{BMS19}, the cyclotomic structure 
    %\begin{align*}
    %\varphi_p:\Sphere_{W(k)}[x_0,x_1,\ldots,x_n]\rightarrow & (\Sphere_{W(k)}[x_0,x_1,\ldots,x_n])^{tC_p} \\
    %\simeq & (\Sphere_{W(k)}[x_0,x_1,\ldots,x_n])^\wedge_p
    %\end{align*}
    %is determined by the Frobenius on $W(k)$ and the assignment that $x_i\mapsto x_i^p$, $0\leq i\leq n$; 
    Consider the map 
    $$
    \phi:\Sphere_{W(k)}[x_0,x_1,\ldots,x_n]\rightarrow\Sphere_{W(k)}[x_0,x_1,\ldots,x_n],
    $$
    where $\phi$ is determined by the Frobenius of $k$ and $\phi(x_i)=x_i^p$ for $0\leq i\leq n$. By \cite[Proposition 11.3]{BMS19}, the map $\phi$ is in fact a morphism of cyclotomic bases. We set 
    $$
    B_n:=\pi_0 \B_n=W(k)[x_0,\ldots,x_n],
    $$ 
    which is a $\delta$-ring whose Frobenius is induced by $\phi$. We regard $B_n$ as a $\B_n$-algebra via the canonical truncation $\B_n\rightarrow \tau_{\leq 0}\B_n=\pi_0\B_n$.
\end{ex}

%\begin{rem}\label{rem:polynomial-base}
%    In Example \ref{ex:polynomial-base}, we can in fact take $k$ to be any perfect ring of characteristic $p$ and also consider $\Sphere_{W(k)}[x_0,x_1,\ldots]$, the spherical polynomial ring with infinitely many variables, and it is still a cyclotomic base. %But in general we do not have $\Sphere_{W(k)}^{tC_p}\simeq \Sphere_{W(k)}$. 
%\end{rem}

\begin{ex}\label{ex:polynomial-base-perfection}
    Let $k$ be a perfect ring of characteristic $p$, and let $\Sphere_{W(k)}$ be the corresponding ring of spherical Witt vectors with coefficients in $k$. Set 
    $$
    \B_n:=\Sphere_{W(k)}[x_0,x_1,\ldots,x_n]=\Sphere_{W(k)}\otimes (\Sigma^{\infty}_+\N)^{\otimes (n+1)}.
    $$
    Since $\phi:\B_n\rightarrow \B_n$ is a morphism of cyclotomic bases, the object $\B_{n,\perf}:=\Sphere_{W(k)}[x_0^{1/p^\infty},\ldots,x_n^{1/p^\infty}]=\varinjlim_\phi \B_n$ is also a cyclotomic base. Note that $(\B_{n,\perf})^\wedge_p$ satisfies the hypothesis of Lemma \ref{lem:spherical-perfect}: in fact, $(\B_{n,\perf})^\wedge_p\simeq \Sphere_{W(k[x_0^{1/p^\infty},x_n^{1/p^\infty}])}$.  
\end{ex}

\begin{lem}[{\cite[Lemma 2.3]{LW22}}]\label{lem:cyclotomic-base}
    Let $E$ be a cyclotomic base. Then the functor $\THH(-/E):\CAlg_E\rightarrow \CAlg_E^{B\T}$ admits a natural lift to a functor 
        $$
        \THH(-/E):\CAlg_E\rightarrow \CycSp_E.
        $$
    Moreover, this lift is functorial in $E$. 
\end{lem}

\begin{proof}
    See \cite[Lemma 2.3]{LW22}.
\end{proof}

\begin{defn}\label{notation:relative-cyclotomic-Frobenius}
    Let $E$ be a cyclotomic base whose underlying spectrum is connective, and let $\THH(-/E):\CAlg_{E}\rightarrow \CycSp_E$ be the corresponding relative topological Hochschild homology. For each prime $\ell$, there is the {\it canonical map} 
    \begin{align*}
        \can:\TC^-(-/E) & \simeq (\THH(-/E)^{hC_\ell})^{h(\T/C_\ell)}  \\
        & = (\THH(-/E)^{hC_\ell})^{h\T} \rightarrow (\THH(-/E)^{tC_\ell})^{h\T},
    \end{align*}
    and when restricted to connective $\Einfty$-$E$-algebras, we have (see \cite[Lemma II.4.2]{NS18})
    $$
    (\THH(-/E)^{tC_\ell})^{h\T}\simeq \TP(-/E)^\wedge_\ell
    $$ 
    We denote the {\it cyclotomic Frobenius of $\THH(-/E)$ at $\ell$} by 
    $$
    \varphi_\ell:\THH(-/E)\rightarrow \THH(-/E)^{tC_\ell}, 
    $$
    which further induces 
    $$
    \varphi_{\ell}^{h\T}:\TC^-(-/E)\rightarrow (\THH(-/E)^{tC_\ell})^{h\T}.
    $$
    We define the {\it topological cyclic homology relative to $E$} as follows:
    $$
    \TC(-/E):=\fib(\TC^-(-/E)\xrightarrow{\prod_{\ell}\varphi_\ell^{h\T}-\can}\prod_\ell (\THH(-/E)^{tC_\ell})^{h\T}).
    $$

    Let $R$ be a connective $E$-algebra, and let $p$ be a prime number. Then $\THH(R/E)$ is connective, and $\THH(R/E)^\wedge_p$ is a $p$-cyclotomic spectrum and $\varphi_\ell$ becomes trivial if $\ell\neq p$. We have 
    $$
    \TC(R/E)^\wedge_\ell:=\fib(\TC^-(R/E)^\wedge_\ell\xrightarrow{\varphi_p^{h\T}-\can}\TP(R/E)^\wedge_\ell).
    $$
    If $\THH(R/E)$ is already $p$-complete, then we have 
    $$
    \TC(R/E)=\fib (\TC^-(R/E)\xrightarrow{\varphi_p^{h\T}-\can}\TP(R/E)).
    $$
\end{defn}

%\begin{rem}\label{rem:spectral-sequences}
%    There are the homotopy fixed point spectral sequence 
%    $$
%    E^{2}_{i,j}=\THH_*(-/E)[v]\Rightarrow \TC^-_{i+j}(-/E),
%    $$
%    and the Tate spectral sequence 
%    $$
%    E^2_{i,j}=\THH_*(-/E)[\sigma^{\pm 1}]\Rightarrow \TP_{i+j}(-/E),
%    $$
%    where $\THH_j(-/E)$ has degree $(0,j)$, $|v|=(-2,0)$, $|\sigma|=(2,0)$, and $\can(v)=\sigma^{-1}$. %The Nygaard filtrations $\calN^{\geq j}$ are defined to be the abutment filtrations of these spectral sequences, and these filtrations are multiplicative. 
    
%    If $\THH_*(-/E)$ is concentrated in even degrees, then the Tate spectral sequence collapses at the $E^2$-page, and let $\Fil^{\geq \star}_{\text{abutment}}\TP_0(-/E)$ denote the corresponding abutment filtration. Since for $i\in \Z$ we have $\Fil^{\geq 2i}_{\text{abutment}}\TP_0(-/E)=\Fil^{\geq 2i+1}_{\text{abutment}}\TP_0(-/E)$, we define the Nygaard filtration $\calN^{\geq \star}\TP_0(-/E)$ on $\TP_0(-/E)$ by setting 
%    $$
%    \calN^{\geq \star}\TP_0(-/E)=\Fil^{\geq 2\star}_{\text{abutment}}\TP_0(-/E).
%    $$
%    For an integer $j$ we denote by $p_j$ the natural projection 
%    $$
%    \calN^{\geq j} \TP_0(-/E)\rightarrow E^2_{-2j,2j}= \THH_{2j}(-/E)[\sigma^{-j}].
%    $$
%    Moreover, in this case we have 
%    $$
%    \gr^*_{\calN}\TP_0(-/E)=\bigoplus_{j\in \Z}\THH_{2j}(-/E)[\sigma^{-j}],
%    $$
%\end{rem}

\section{A Reminder on Relatively Quasiregular Semiperfectoids}\label{section:a-reminder-on-qrsp-relative}

We will study the relative topological Hochschild homology $\THH(R/\Sphere_{W(k)[x_0,x_1,\ldots,x_n]})$, where the map $W(k)[x_0,x_1,\ldots,x_n]\rightarrow R$ yields a (pre-prismatic) relatively quasiregular semiperfectoid $\delta$-pair in the sense of \cite[Definition 9.2]{AKN23}. In this section, we recall notions concerning relatively quasiregular semiperfectoids from \cite{AKN23} and \cite{BS22}, and we also recall key results on $\THH$ for quasiregular semiperfectoid rings from \cite{BS22} and \cite{BMS19}. Fix a prime $p$. 

\begin{defn}[{\cite[Definition 3.2]{BS22}}]\label{defn:distinguished-ideals}
    Let $A$ be a ($p$-typical) $\delta$-ring. An ideal $I\subset A$ is called {\it distinguished} if the following conditions are satisfied:
    \begin{itemize}
        \item[(1)] As an $A$-module, $I$ is invertible.  

        \item[(2)] The ring $A/I$ has bounded $p$-power torsion. 

        \item[(3)] We have $p\in IA^\wedge_{(p,I)}+\phi(I)A^\wedge_{(p,I)}$.
    \end{itemize}
    Equivalently, $I\subset A$ is distinguished if and only if $(A^{\wedge}_{(p,I)},IA^{\wedge}_{(p,I)})$ is a bounded prism in the sense of \cite[Definition 3.2]{BS22}.
\end{defn}

In this note, we will mainly consider the $\delta$-ring $B_n=W(k)[x_0,\ldots,x_n]$, where $k$ is a finite field and $B_n$ is equipped with the $\delta$-structure determined by the canonical $\delta$-structure on $W(k)$ and by $\delta(x_i)=0$ for $0\leq i\leq n$.

\begin{ex}\label{ex:distinguished-ideals}
    Let $k$ be a perfect field of characteristic $p$, and let $B_n=W(k)[x_0,x_1,\ldots,x_n]$. We have the following important examples of distinguished ideals $I$ in $B_n$:
    \begin{itemize}
        \item[(a)] Take $I=(p)$.
        
        \item[(b)] Take $I=(E(x_0))$, where $E(x_0)\in W(k)[x_0]\subset B_n$ is an Eisenstein polynomial.  

        \item[(c)] Take $I=([p]_{x_0})$, where $[p]_{x_0}=\frac{x_0^{p}-1}{x_0-1}\in W(k)[x_0]$.
    \end{itemize}
\end{ex}

\begin{defn}[{\cite[Definition 9.2]{AKN23}, \cite[Definition 15.1]{BS22}}]\label{defn:qrsp-relative}
    Let $A$ be a $\delta$-ring and $R$ a commutative ring. A ring homomorphism $A\rightarrow R$ is called {\it relatively quasiregular semiperfectoid} if the following conditions are satisfied:
    \begin{itemize}
        \item[(1)] The ring $R$ has bounded $p$-power torsion and is $p$-complete. 

        \item[(2)] The ring homomorphism $A\rightarrow  R$ is surjective. 

        \item[(3)] The ideal $\ker(A\rightarrow R)$ contains a distinguished ideal $I$ of $A$, and $L\Omega^1_{R/(A/I)}\in \rmD(R)$ has $p$-complete Tor-amplitude concentrated in homological degree $1$.
    \end{itemize} 
\end{defn}

Note that Definition \ref{defn:qrsp-relative} is a restricted version of \cite[Definition 9.2]{AKN23} (or \cite[Definition 15.1]{BS22}). %and therefore not suitable for theoretical development. 
We introduce this definition only for ease of exposition.

\begin{ex}\label{ex:qrsp}
    Let $S$ be a quasiregular semiperfectoid ring in the sense of \cite[Definition 4.20]{BMS19}, that is, $S$ satisfies the following conditions (see also \cite[Remark 4.22]{BMS19} and \cite[Lemma 4.25]{BMS19}):
    \begin{itemize}
        \item[(1)] The ring $S$ has bounded $p$-power torsion and is $p$-complete.

        \item[(2)] There exists a surjective map $R\rightarrow S$ with $R$ perfectoid such that $L\Omega^1_{S/R}\in \rmD(S)$ has $p$-complete Tor-amplitude concentrated in homological degree $1$.
    \end{itemize}
    Then the map $W(R^\flat)\rightarrow S$ is relatively quasiregular semiperfectoid in the sense of Definition \ref{defn:qrsp-relative}. Note that, by \cite[Definition 9.2]{AKN23} and \cite[Example 9.4]{AKN23}, we should regard $\Z_p\rightarrow S$ as relatively quasiregular semiperfectoid.
\end{ex}

\begin{ex}\label{ex:qrsp-relative-polynomial}
    Let $K$ be a finite extension of $\Q_p$, and let $\calO_K$ be the ring of integers of $K$. Let $k$ be the residue field, and let $\varpi\in \calO_K$ be a uniformizer. %and $E_K(x_0)\in W(k)[x_0]$ a minimal polynomial of $\varpi$. 
    Then the following statements hold:
    \begin{itemize}
        \item[(a)] The ring homomorphism $W(k)[x_0]\xrightarrow{x_0\mapsto \varpi}\calO_K$ is relatively quasiregular semiperfectoid. 

        \item[(b)] In general, for any integer $n\geq 0$, the ring homomorphism $W(k)[x_0,\ldots,x_n]\xrightarrow{x_i\mapsto \varpi,0\leq i\leq n}\calO_K$ is relatively quasiregular semiperfectoid.
    \end{itemize}
\end{ex}

\begin{lem}\label{lem:base-change-for-qrsp-relative}
    Let $A$ be a $\delta$-ring and $R$ a commutative ring, and let $A\rightarrow R$ be relatively quasiregular semiperfectoid. Suppose that $A\rightarrow A'$ is a $p$-completely flat map of $\delta$-rings. Then the map $A'\rightarrow R'$, where $R':=R\widehat{\otimes}^L_{A}A'$ is relatively quasiregular semiperfectoid.
\end{lem}

\begin{proof}
    Since $A\rightarrow A'$ is $p$-completely flat, by \cite[Corollary 4.8]{BMS19}, $R'$ has bounded $p$-power torsion and is $p$-complete. It is clear that $A'\rightarrow R'$ is surjective. Let $I\subset A$ be a distinguished ideal contained in $\ker(A\rightarrow R)$. Then $IA'\subset A'$ is a distinguished ideal contained in $\ker(A'\rightarrow R')$. Finally, we have 
    \begin{align*}
    L\Omega^1_{R\otimes_A^LA'/(A'/IA')}
    & \simeq L\Omega^1_{R\otimes^L_{A/I}(A'/IA')/(A'/IA')}\\ 
    & \simeq L\Omega^1_{R/(A/I)}\otimes^L_{A/I}(A'/IA') \\
    & \simeq L\Omega^1_{R/(A/I)}\otimes^L_R(R\otimes^L_{A}A'),
    \end{align*}
    and since $L\Omega^1_{R/(A/I)}\in \rmD(R)$ has $p$-complete Tor-amplitude concentrated in homological degree $1$, $L\Omega^1_{R'/(A'/IA')}\in \rmD(R')$ has $p$-complete Tor-amplitude concentrated in homological degree $1$.
\end{proof}

\begin{cor}\label{cor:perfection-and-qrsp}
    Let $A$ be a $\delta$-ring whose Frobenius $\phi:A\rightarrow A$ is $p$-completely flat, and let $R$ be a commutative ring. Assume that $A\rightarrow R$ is relatively quasiregular semiperfectoid. Set $A_\perf:=\varinjlim_{\phi}A$ with the canonical map $A\rightarrow A_\perf$.  Then $R_\perf:=R\widehat{\otimes}^L_{A}A_\perf$ is a quasiregular semiperfectoid ring.
\end{cor}

\begin{proof}
    By Lemma \ref{lem:base-change-for-qrsp-relative}, the map $A_\perf\rightarrow R_\perf$ is relatively quasiregular semiperfectoid. Now let $I\subset A_\perf$ be a distinguished ideal contained in $\ker(A_\perf\rightarrow R_\perf)$. Then $((A_\perf)^\wedge_{(p,I)},I(A_\perf)^\wedge_{(p,I)})$ is a perfect prism and, by \cite[Theorem 3.10]{BS22} $(A_\perf/I)^\wedge_p$ is a perfectoid ring. Since $R_\perf$ is $p$-complete, we have a surjective map $(A_\perf/I)^\wedge_p\rightarrow R_\perf$. Note that $L\Omega^1_{R_\perf/(A_\perf/I)}$ and $L\Omega^1_{R_\perf/(A_\perf/I)^\wedge_p}$ have the same $p$-complete Tor-amplitude. Therefore $R_\perf$ is a quasiregular semiperfectoid ring.  
\end{proof}

We close this section by briefly recalling prismatic cohomology and $\THH$ for quasiregular semiperfectoid rings. For a comprehensive account of prismatic cohomology, %(and a general theory of prismatic cohomology relative to $\delta$-rings), 
see, for example, \cite[\S 1]{AKN23} and \cite[Theorem 1.8, Theorem 1.14]{BS22}. Here we only fix notation. 

\begin{notation}[Relative Prismatic Cohomology, {\cite[\S 1]{AKN23}}]\label{notation:relative-prismatic-cohomology}
    Let $(A,I)$ be a bounded prism and $R$ an $A/I$-algebra. The theory of prismatic cohomology of $R$ relative to $(A,I)$ produces the following data:
    $$
    \underline{\Prism}_{R/A}=(\Prism^{[\star]}_{R/A}\{*\},\calN^{\geq \star}\widehat{\Prism}_{R/A}^{(1)}\{*\},c,\phi_{/A}),
    $$
    where 
    \begin{itemize}
        \item[(a)] $\Prism_{R/A}$ is an $\Einfty$-$A$-algebra, which is referred to as the {\it derived prismatic cohomology of $R$ relative to $(A,I)$};

        \item[(b)] $\Prism_{R/A}\{i\}$, $i\in \Z$, are $\Prism_{R/A}$-modules (with $\Prism_{R/A}\{0\}=\Prism_{R/A}$) which are graded multiplicative with respect to $i\in \Z$; we will refer to $\Prism_{R/A}\{i\}$ as the {\it $i$-th Breuil-Kisin twist of $\Prism_{R/A}$};

        \item[(c)] for $i\in \Z$, $\Prism^{[\star]}_{R/A}\{i\}$ is a complete tower
        $$
        \cdots\rightarrow \Prism^{[n+1]}_{R/A}\{i\}\rightarrow \Prism^{[n]}_{R/A}\{i\}\rightarrow \Prism^{[n-1]}_{R/A}\{i\}\rightarrow \cdots,
        $$
        which is multiplicative in $i,n\in \Z$, with $\Prism^{[0]}_{R/A}\{i\}=\Prism_{R/A}\{i\}$ and underlying object $\Prism_{R/A}\{i\}[1/I]$;
        we will refer to $\Prism^{[\star]}_{R/A}\{i\}$ as the {\it Hodge-Tate tower on $\Prism_{R/A}\{i\}$}; in fact, $\Prism_{R/A}^{[n]}\{i\}=\Prism_{R/A}\{i\}\otimes_A I^n$, and 
        $$
        \gr^n\Prism^{[\star]}_{R/A}\{i\}\simeq \overline{\Prism}_{R/A}\{i+n\}, 
        $$
        where $\overline{\Prism}_{R/A}\{i+n\}$ is the {\it $(i+n)$-th Breuil-Kisin twisted Hodge-Tate cohomology of $R$ relative to $(A,I)$}; note that for $i\in \Z$, $\overline{\Prism}_{R/A}\{i\}$ is equipped with the {\it conjugate filtration $\Fil^\cj_{\leq \star}\overline{\Prism}_{R/A}\{i\}$ relative to $(A,I)$}
        $$
        \cdots\rightarrow \Fil^{\cj}_{\leq n-1}\overline{\Prism}_{R/A}\{i\}\rightarrow \Fil^{\cj}_{\leq n}\overline{\Prism}_{R/A}\{i\}\rightarrow \Fil^{\cj}_{\leq n+1}\overline{\Prism}_{R/A}\{i\}\rightarrow \cdots
        $$
        which is increasing, exhaustive, and multiplicative in $i,n\in \Z$ and is zero for $\star<0$, with 
        $$
        \gr^\cj_n \overline{\Prism}_{R/A}\{i\}\simeq L\widehat{\Omega}^n_{R/(A/I)}[-n]\{i- n\}, \quad \quad n\geq 0.
        $$

        \item[(d)] for $i\in \Z$, $\calN^{\geq \star}\widehat{\Prism}_{R/A}^{(1)}\{i\}$ is a complete filtration of $A$-modules 
        $$
        \cdots\rightarrow \calN^{\geq n+1}\widehat{\Prism}_{R/A}^{(1)}\{i\}\rightarrow \calN^{\geq n}\widehat{\Prism}_{R/A}^{(1)}\{i\}\rightarrow \calN^{n-1}\widehat{\Prism}_{R/A}^{(1)}\{i\}\rightarrow \cdots,
        $$
        which is multiplicative in $i,n\in \Z$, and is constant for $\star\leq 0$; we refer to $\calN^{\geq \star}\widehat{\Prism}_{R/A}^{(1)}\{i\}$ as the {\it Nygaard filtration on the Nygaard-completed, Frobenius twisted, $i$-th Breuil-Kisin twisted prismatic cohomology $\widehat{\Prism}_{R/A}^{(1)}\{i\}$};

        \item[(e)] for $i\in \Z$, $c:\Prism_{R/A}\{i\}\rightarrow \widehat{\Prism}_{R/A}^{(1)}\{i\}$ is a $\phi_A$-semilinear map;

        \item[(f)] for $i\in \Z$, $\phi_{/A}:\calN^{\geq \star}\widehat{\Prism}_{R/A}^{(1)}\{i\}\rightarrow \Prism^{[\star-i]}_{R/A}\{i\}$ is the {\it filtered relative Frobenius map}; by taking the associated graded at $n\geq 0$, we get the map $\gr^n\phi_{/A}:\gr^n_\calN\widehat{\Prism}^{(1)}_{R/A}\{i\}\rightarrow \gr^n\Prism_{R/A}^{[\star-i]}\{i\}\simeq \overline{\Prism}_{R/A}\{n\}$, which induces an equivalence
        $$
        \gr^n_\calN\widehat{\Prism}_{R/A}^{(1)}\{i\}\xrightarrow{\sim} \Fil_{\leq n}^\cj \overline{\Prism}_{R/A}\{n\}.
        $$
    \end{itemize}
\end{notation}

\begin{thm}[{\cite[\S 7.2]{BMS19}}, {\cite[\S 13]{BS22}}]\label{thm:THH-for-qrsp}
    Let $S$ be a quasiregular semiperfectoid ring equipped with a ring homomorphism $R\rightarrow S$ where $R$ is a perfectoid ring with the Fontaine map $\theta:W(R^\flat)\rightarrow R$, and regard $(W(R^\flat),\ker\theta)$ as a prism. For $i\in \Z$, let $\calN^{\geq \star}\widehat{\Prism}^{(1)}_{S/W(R^\flat)}\{i\}$ denote the Nygaard filtration on the Nygaard-completed, Frobenius-twisted, $i$-th Breuil-Kisin twisted prismatic cohomology $\widehat{\Prism}_{S/W(R^\flat)}^{(1)}\{i\}$, and let $\gr_\calN^*\widehat{\Prism}_{S/W(R^\flat)}^{(1)}\{i\}$ denote the associated graded. We also regard $S$ naturally as an $\Sphere_{W(R^\flat)}$-algebra. Note that $\THH(S/\Sphere_{W(R^\flat)})$ is $p$-complete by \cite[Lemma 3.6]{KN22}. Then the following statements hold:
    \begin{itemize}
        \item[(1)] The homotopy groups of $\THH(S/\Sphere_{W(R^\flat)})$ are concentrated in even degrees, and there is a canonical isomorphism of graded rings 
        $$
        \pi_{2*} \THH(S/\Sphere_{W(R^\flat)})\simeq \gr^*_\calN\widehat{\Prism}^{(1)}_{S/W(R^\flat)}.
        $$

        \item[(2)] The homotopy groups of $\TC^-(S/\Sphere_{W(R^\flat)})$ are concentrated in even degrees, and there is a canonical isomorphism of graded rings 
        $$
        \pi_{2*}\TC^-(S/\Sphere_{W(R^\flat)})\simeq \bigoplus_{*\in \Z\ \text{graded}}\calN^{\geq *}\widehat{\Prism}^{(1)}_{S/W(R^\flat)}\{*\}.
        $$

        \item[(3)] The homotopy groups of $\TP(S/\Sphere_{W(R^\flat)})$ are concentrated in even degrees, and there is a canonical isomorphism of graded rings 
        $$
        \pi_{2*}\TP(S/\Sphere_{W(R^\flat)})\simeq \bigoplus_{*\in \Z \ \text{graded}}\widehat{\Prism}^{(1)}_{S/W(R^\flat)}\{*\}.
        $$

        \item[(4)] Under the identifications in (2) and (3), the canonical map $\can:\TC^-(S/\Sphere_{W(R^\flat)})\rightarrow \TP(S/\Sphere_{W(R^\flat)})$ induces on homotopy groups the canonical inclusions of the Nygaard filtration.

        \item[(5)] Under the identifications in (2) and (3), the cyclotomic Frobenius $\varphi_p^{h\T}:\TC^-(S/\Sphere_{W(R^\flat)})\rightarrow \TP(S/\Sphere_{W(R^\flat)})$ induces on homotopy groups the following map 
        $$
        \bigoplus_{*\in \Z\ \text{graded}}\calN^{\geq *}\widehat{\Prism}^{(1)}_{S/W(R^\flat)}\{*\}\xrightarrow{c\circ\phi_{/W(R^\flat)}}\bigoplus_{*\in \Z\ \text{graded}}\widehat{\Prism}^{(1)}_{S/W(R^\flat)}\{*\},
        $$
        where $c$ and $\phi_{/W(R^\flat)}$ are maps in Notation \ref{notation:relative-prismatic-cohomology}, (e) and (f), respectively. 

        %Consider the following commutative diagram:
        %$$
        %\xymatrix@R=50pt@C=50pt{
        %\TC^-(S/\Sphere_{W(R^\flat)}) \ar[r]^{\varphi_p^{h\T}} \ar[d] & \TP(S/\Sphere_{W(R^\flat)}) \ar[d] \\
        %\THH(S/\Sphere_{W(R^\flat)}) \ar[r]^{\varphi_p} & %\THH(S/\Sphere_{W(R^\flat)})^{tC_p}.
        %}
        %$$
        %With the identifications in (1), (2), and (3), by taking the homotopy groups of the above commutative diagram we obtain the following commutative diagram of graded rings:
        %$$
        %\xymatrix@R=50pt@C=50pt{
        %\bigoplus_{*\in \Z\ \text{graded}}\calN^{\geq *}\widehat{\Prism}^{(1)}_{S/W(R^\flat)}\{*\}
        %\ar[r]^{c\circ \phi_{/A}} \ar[d] & \bigoplus_{*\in \Z\ %\text{graded}}\widehat{\Prism}^{(1)}_{S/W(R^\flat)}\{*\} \ar[d] \\
        %\gr^*_\calN\widehat{\Prism}^{(1)}_{S/W(R^\flat)} \ar[r]^-{c\circ %\gr^*\phi_{/A}} & \bigoplus_{*\in \Z\ \text{graded}}(\widehat{\Prism}^{(1)}_{S/W(R^\flat)}\{*\}/\phi(\ker\theta)).
        %}
        %$$
    \end{itemize}
\end{thm}

\begin{proof}
    See \cite[\S 7.2]{BMS19} and \cite[\S 13]{BS22}. Note that the statements here are somewhat awkward and not immdiately those presented in {\cite[\S 7.2]{BMS19}} and {\cite[\S 13]{BS22}}. However, since $(W(R^\flat),\ker\theta)$ is a perfect prism, there is a canonical isomorphism 
    $$
    \Prism_S\simeq \Prism_{S/W(R^\flat)}\xrightarrow{\sim}\Prism^{(1)}_{S/W(R^\flat)}:=\Prism_{S/W(R^\flat)}\widehat{\otimes}^L_{W(R^\flat),\phi}W(R^\flat),
    $$
    inducing isomorphisms $\calN^{\geq \star}\widehat{\Prism}_{S}\{*\}\simeq \calN^{\geq \star}\widehat{\Prism}^{(1)}_{S/W(R^\flat)}\{*\}$ (see \cite[\S 12, \S 15]{BS22}). On the other hand, by Example \ref{ex:spherical-Witt-vectors}, the canonical map $\THH(S)^\wedge_p\rightarrow \THH(S/\Sphere_{W(R^\flat)})$ is an equivalence. Therefore the statements of this theorem are canonically equivalent to those presented in {\cite[\S 7.2]{BMS19}} and {\cite[\S 13]{BS22}}. We state this theorem in this style so that we can use it directly in the proof of Theorem \ref{thm:THH-for-qrsp-relative-polynomial}.
\end{proof}

\section{\texorpdfstring{$\THH$}{THH} for Quasiregular Semiperfectoids Relative to \texorpdfstring{$\Sphere_{W(k)}[x_0,x_1,\ldots,x_n]$}{}}\label{section:THH-for-qrsp-relative-polynomial}

In this section, we study the relative topological Hochschild homology of the form $\THH(R/\Sphere_{W(k)}[x_0,x_1,\ldots,x_n])$, where the map $W(k)[x_0,x_1,\ldots,x_n]\rightarrow R$ is relatively quasiregular semiperfectoid in the sense of Definition \ref{defn:qrsp-relative}; this definition is a restricted version of \cite[Definition 9.2]{AKN23} or \cite[Definition 15.1]{BS22}.

\begin{thm}[\cite{Krause23}]\label{thm:THH-for-qrsp-relative-polynomial}
    Let $k$ be a perfect ring of characteristic $p$, and let $\Sphere_{W(k)}$ be the corresponding ring of spherical Witt vectors with coefficients in $k$. Set 
    $$
    \B_n:=\Sphere_{W(k)}[x_0,x_1,\ldots,x_n]=\Sphere_{W(k)}\otimes (\Sigma^{\infty}_+\N)^{\otimes (n+1)}.
    $$    
    We also set $B_n=W(k)[x_0,x_1,\ldots,x_n]$ equipped with the $\delta$-structure determined by the canonical $\delta$-structure on $W(k)$ and by $\delta(x_i)=0$ for $0\leq i\leq n$. Let $R$ be a commutative ring with a relatively quasiregular semiperfectoid map $B_n\rightarrow R$; in other words, the following conditions are satisfied:
    \begin{itemize}
        \item[(i)] The ring $R$ has bounded $p$-power torsion and is $p$-complete. 

        \item[(ii)] The ring homomorphism $B_n\rightarrow  R$ is surjective. 

        \item[(iii)] The ideal $\ker(B_n\rightarrow R)$ contains a distinguished ideal $I$ of $B_n$, and $L\Omega^1_{R/(B_n/I)}\in \rmD(R)$ has $p$-complete Tor-amplitude concentrated in homological degree $1$.
    \end{itemize} 
    Note that $\THH(R/\B_n)$ is $p$-complete by \cite[Lemma 3.6]{KN22}.
    Now let $B_n^\wedge$ denote the $(p,I)$-completion of $B_n$, and regard $(B_n^\wedge, IB_n^\wedge)$ as a prism. Then the following statements hold:
    \begin{itemize}
        \item[(1)] The homotopy groups of $\THH(R/\B_n)$ are concentrated in even degrees, and there is a canonical isomorphism of graded rings 
        $$
        \pi_{2*} \THH(R/\B_n)\simeq \gr^*_\calN\widehat{\Prism}^{(1)}_{R/B_n^\wedge}.
        $$

        \item[(2)] The homotopy groups of $\TC^-(R/\B_n)$ are concentrated in even degrees, and there is a canonical isomorphism of graded rings 
        $$
        \pi_{2*}\TC^-(R/\B_n)\simeq \bigoplus_{*\in \Z\ \text{graded}}\calN^{\geq *}\widehat{\Prism}^{(1)}_{R/B_n^\wedge}\{*\}.
        $$

        \item[(3)] The homotopy groups of $\TP(R/\B_n)$ are concentrated in even degrees, and there is a canonical isomorphism of graded rings 
        $$
        \pi_{2*}\TP(R/\B_n)\simeq \bigoplus_{*\in \Z\ \text{graded}}\widehat{\Prism}^{(1)}_{R/B_n^\wedge}\{*\}.
        $$

        \item[(4)] Under the identifications in (2) and (3), the canonical map $\can:\TC^-(R/\B_n)\rightarrow \TP(R/\B_n)$ induces on homotopy groups the canonical inclusions of the Nygaard filtration.

        \item[(5)] Under the identifications in (2) and (3), the cyclotomic Frobenius $\varphi_p^{h\T}:\TC^-(R/\B_n)\rightarrow \TP(R/\B_n)$ induces on homotopy groups the following map 
        $$
        \bigoplus_{*\in \Z\ \text{graded}}\calN^{\geq *}\widehat{\Prism}^{(1)}_{R/B_n^\wedge}\{*\}\xrightarrow{c\circ\phi_{/B_n^\wedge}}\bigoplus_{*\in \Z\ \text{graded}}\widehat{\Prism}^{(1)}_{R/B_n^\wedge}\{*\},
        $$
        where $c$ and $\phi_{/B_n^\wedge}$ are maps in parts (e) and (f) of Notation \ref{notation:relative-prismatic-cohomology}, respectively.

        %Consider the following commutative diagram:
        %$$
        %\xymatrix@R=50pt@C=50pt{
        %\TC^-(R/\B_n) \ar[r]^{\varphi_p^{h\T}} \ar[d] & %\TP(R/\B_n) \ar[d] \\
        %\THH(R/\B_n) \ar[r]^{\varphi_p} & \THH(R/\B_n)^{tC_p}.
        %}
        %$$
        %With the identifications in (1), (2), and (3), by taking the homotopy groups of the above commutative diagram we obtain the following commutative diagram of graded rings:
        %$$
        %\xymatrix@R=50pt@C=50pt{
        %\bigoplus_{*\in \Z\ \text{graded}}\calN^{\geq %*}\widehat{\Prism}^{(1)}_{R/B_n^\wedge}\{*\} 
        %\ar[r]^{c\circ \phi_{/B_n^\wedge}} \ar[d] & \bigoplus_{*\in \Z\ %\text{graded}}\widehat{\Prism}^{(1)}_{R/B_n^\wedge}\{*\} \ar[d] \\
        %\gr^*_\calN\widehat{\Prism}^{(1)}_{R/B_n^\wedge} \ar[r]^-{c\circ %\gr^*\phi_{/B_n^\wedge}} & \bigoplus_{*\in \Z\ \text{graded}}%(\widehat{\Prism}^{(1)}_{R/B_n^\wedge}\{*\}/\phi(I)).
        %}
        %$$
    \end{itemize}
\end{thm}

\begin{proof}
    By Theorem \ref{thm:THH-for-qrsp}, these statements are almost immediate. 
    Let $\B_{n,\perf}=\Sphere_{W(k)}[x_0^{1/p^\infty},\ldots,x_n^{1/p^\infty}]$ and $B_{n,\perf}=W(k)[x_0^{1/p^\infty},\ldots,x_n^{1/p^\infty}]$. Let $B_{n,\perf}^\wedge$ denote the $(p,I)$-completion of $B_{n,\perf}$. Regard $(B_{n,\perf}^\wedge, IB_{n,\perf}^\wedge)$ as a prism. We argue by descent from Theorem \ref{thm:THH-for-qrsp} along $(B_n^\wedge, IB_n^\wedge)\rightarrow (B_{n,\perf}^\wedge, IB_{n,\perf}^\wedge)$ and $\B_n\rightarrow \B_{n,\perf}$.

    By \cite[Theorem 1.2 (7)]{AKN23}, the prismatic data $\underline{\Prism}_{R/A}$ satisfies base change in $A$, and therefore 
    \begin{equation}\label{equation:prismatic-base-change}
         ((\calN^{\geq \star}\widehat{\Prism}^{(1)}_{R/B_n^\wedge}\{*\})\widehat{\otimes}^L_{B_n}B_{n,\perf})^\wedge_\calN\simeq \calN^{\geq \star}\widehat{\Prism}^{(1)}_{R\widehat{\otimes}^L_{B_n}B_{n,\perf}/B_{n,\perf}^\wedge}\{*\},
    \end{equation}
    where $\widehat{\otimes}^L$ is the derived $p$-completed tensor product. Note that $\phi:B_n\rightarrow B_n$ is $p$-completely faithfully flat, and by Corollary \ref{cor:perfection-and-qrsp}, $R\widehat{\otimes}^L_{B_n}B_{n,\perf}$ is a quasiregular semiperfectoid ring. 

    On the other hand, we have 
    %notice that $\THH(R/\B_n)$ is an (non-equivariant) $R$-algebra, 
    \begin{align*}
         (\THH(R/\B_n)\otimes_{\B_n}\B_{n,\perf})^\wedge_p 
        & \simeq \THH(R\otimes^L_{\B_n}\B_{n,\perf}/\B_{n,\perf})^\wedge_p \\
        & \simeq \THH(R\widehat{\otimes}^L_{B_n}B_{n,\perf}/\B_{n,\perf}),
    \end{align*}
    and since $\B_n\rightarrow \B_{n,\perf}$ is flat with $B_n\rightarrow B_{n,\perf}$ being $p$-completely faithfully flat, we have 
    \begin{equation}\label{equation:THH-base-change}
    \pi_*(\THH(R/\B_n))\widehat\otimes^L_{B_n}B_{n,\perf}\simeq \pi_*\THH(R\widehat{\otimes}^L_{B_n}B_{n,\perf}/\B_{n,\perf}).
    \end{equation}
    By Theorem \ref{thm:THH-for-qrsp} (1), (\ref{equation:prismatic-base-change}), (\ref{equation:THH-base-change}), and flat descent, we immediately verify statement (1).
    
    Next, we consider $(\TC^-(R/\B_n)\otimes_{\B_n}\B_{n,\perf})^\wedge_p\rightarrow \TC^-(R\widehat{\otimes}^L_{B_n}B_{n,\perf}/\B_{n,\perf})$. Let $\{E^{*,*}_2\}$ denote the homotopy fixed points spectral sequence for $\TC^-(R/\B_n)$, and let $\{E^{*,*}_{2,\perf}\}$ denote the homotopy fixed points spectral sequence for $\TC^-(R\widehat{\otimes}^L_{B_n}B_{n,\perf}/\B_{n,\perf})$. By (\ref{equation:THH-base-change}), we have 
    $$
    \{E^{*,*}_2\widehat{\otimes}^L_{B_n}B_{n,\perf}\}\xrightarrow{\sim}\{E^{*,*}_{2,\perf}\}, 
    $$
    and therefore the canonical map 
    $$\pi_*(\TC^-(R/\B_n))\widehat{\otimes}^L_{B_n}B_{n,\perf}\rightarrow \pi_*\TC^-(R\widehat{\otimes}^L_{B_n}B_{n,\perf}/\B_{n,\perf})$$ becomes an isomorphism after Nygaard completion. Together with Theorem \ref{thm:THH-for-qrsp} (2), we can verify statement (2): for $i\in \Z$, both $\calN^{\geq i}\widehat{\Prism}_{R/B_n^\wedge}^{(1)}\{i\}$ and $\pi_{2i}(\TC^-(R/\B_n))$ can be identified with the totalization of the following cosimplicial object 
    $$
    ((\calN^{\geq i}\widehat{\Prism}^{(1)}_{R\widehat{\otimes}^L_{B_n}B_{n,\perf}/B_{n,\perf}^\wedge}\{i\})\widehat{\otimes}^L_{B_n}B_{n,\perf}^{\widehat{\otimes}_{B_n}^L\bullet})^\wedge_\calN.
    $$

    The remaining statements of this theorem follow the same logic.
\end{proof}

\begin{rem}\label{rem:THH-and-prismatic-rel-delta-rings}
    By definition, $\THH(R/\B_n)$, $\TC^-(R/\B_n)$, and $\TP(R/\B_n)$ are independent of the choice of a distinguished ideal $I\subset \ker(B_n\rightarrow R)$. In fact, with the theory of prismatic cohomology relative to $\delta$-rings developed in \cite{AKN23}, we can simply state, for example, that there is a canonical isomorphism of graded rings    
    $$ 
     \pi_{2*}\TC^-(R/\B_n)\simeq \bigoplus_{*\in \Z}\calN^{\geq *}\widehat{\Prism}^{(1)}_{R/B_n}\{*\}[2*].
    $$
    See \cite[Theorem 1.2]{AKN23} and \cite[p. 49, (\$)]{AKN23}. However, in practice we usually specify a choice of distinguished ideal $I\subset \ker(B_n\rightarrow R)$ and use the corresponding conjugate filtration (see Notation \ref{notation:relative-prismatic-cohomology} (c), and also \cite[Warning 3.22]{AKN23}). 
\end{rem}

\begin{rem}\label{rem:THH-and-qrsp-rel-AKN}
    Theorem \ref{thm:THH-for-qrsp-relative-polynomial} is certainly not stated in its most theoretical generality. If we use \cite[Definition 9.2]{AKN23} instead of Definition \ref{defn:qrsp-relative}, then in Theorem \ref{thm:THH-for-qrsp-relative-polynomial} we do not need to assume that $B_n\rightarrow R$ is surjective. Further, as announced in \cite[\S 1.3]{AKN23}, there exist motivic filtrations (using \cite[Proposition 9.14]{AKN23} and the idea of \cite[\S 7.3]{BMS19}) for the functors
    $\TC^-(-/\B_n)^\wedge_p$, $\TP(-/\B_n)^\wedge_p$, and $\TC(-/\B_n)^\wedge_p$ defined on $\QSyn_{B_n}$, where $\QSyn_{B_n}$ is defined in \cite[Definition 9.1]{AKN23}: for example, for any $R\in \QSyn_{B_n}$, $\TC^-(R/\B_n)^\wedge_p$ admits a functorial complete and exhaustive descending multiplicative $\Z$-indexed filtration with $\gr^i\TC^-(R/\B_n)^\wedge_p\simeq \calN^{\geq i}\widehat{\Prism}^{(1)}_{R/B_n}\{i\}[2i]$. A more advanced statement allowing bases other than spectra of the form $\B_n$ has been announced in \cite{Krause23}.
\end{rem}

We apply Theorem \ref{thm:THH-for-qrsp-relative-polynomial} in the following more concrete situation. 

\begin{cor}\label{cor:THH-for-qrsp-relative-polynomial}
    Let $k$ be a perfect ring of characteristic $p$, and let $\Sphere_{W(k)}$ be the corresponding ring of spherical Witt vectors with coefficients in $k$. Set 
    $$
    \B_n:=\Sphere_{W(k)}[x_0,x_1,\ldots,x_n]=\Sphere_{W(k)}\otimes (\Sigma^{\infty}_+\N)^{\otimes (n+1)}.
    $$    
    We also set $B_n=W(k)[x_0,x_1,\ldots,x_n]$ and equip it with the $\delta$-structure determined by the canonical $\delta$-structure on $W(k)$ and by $\delta(x_i)=0$ for $0\leq i\leq n$. Let $I=(d)\subset B_n$ be a distinguished principal ideal. Let $B_n^\wedge$ denote the $(p,I)$-completion of $B_n$. Let $R=B_{n}/(I,f_1,\ldots,f_r)$ where $f_1,\ldots,f_r\in B_n$ form a Koszul regular sequence in $\overline{B}_n:=B_{n}/I$, and let $\overline{J}=(I,f_1,\ldots,f_r)/I$. Then the following statements hold:
    \begin{itemize}
        \item[(1)] The homotopy groups of $\THH(\overline{B}_n/\B_n)$ are concentrated in even degrees, and there is a canonical isomorphism of graded rings 
        $$
        \pi_{2*} (\THH(\overline{B}_n/\B_n))\simeq \overline{B}_n[u],
        $$
        where $u\in \pi_{2} (\THH(\overline{B}_n/\B_n))=I/I^2$ is a generator.
        
        \item[(2)] There is a $\T$-equivariant cofiber sequence 
        $$
        \THH(R/\B_n)[2]\xrightarrow{u}\THH(R/\B_n)\rightarrow \THH(R/\overline{B}_n),
        $$
        inducing an exact sequence of graded groups 
        $$
        0\rightarrow \pi_{*-2}\THH(R/\B_n)\xrightarrow{u} \pi_{*}\THH(R/\B_n)\rightarrow \pi_*\THH(R/\overline{B}_n)\rightarrow 0.
        $$
        Note that $\pi_{2*}\THH(R/\overline{B}_n)\simeq R\langle t'_{1},\ldots,t'_{r}\rangle$, where $R\langle t'_{1},\ldots,t'_{r}\rangle$ is the divided power polynomial on $\{t'_{i}=f_i+\overline{J}^2\in \pi_2\THH(R/\overline{B}_n)=\overline{J}/\overline{J}^2\}_{1\leq i\leq r}$.

        \item[(3)] The homotopy groups of $\TC^-(R/\B_n)$ are concentrated in even degrees, and there is a canonical isomorphism of graded rings 
        $$
        \pi_{2*}\TC^-(R/\B_n)\simeq \bigoplus_{*\in \Z\ \text{graded}}\calN^{\geq *}B_n^\wedge\{\frac{\phi(f_1)}{\phi(d)},\ldots,\frac{\phi(f_r)}{\phi(d)}\}^\wedge_{(p,\calN)}\{*\}.
        $$

        \item[(4)] The homotopy groups of $\TP(R/\B_n)$ are concentrated in even degrees, and there is a canonical isomorphism of graded rings 
        $$
        \pi_{2*}\TP(R/\B_n)\simeq \bigoplus_{*\in \Z\ \text{graded}}B^\wedge_n\{\frac{\phi(f_1)}{\phi(d)},\ldots,\frac{\phi(f_r)}{\phi(d)}\}^\wedge_{(p,\calN)}\{*\}.
        $$
    \end{itemize}
\end{cor}

\begin{proof}
    Statements $(1)$, $(3)$ and $(4)$ are immediate from Theorem \ref{thm:THH-for-qrsp-relative-polynomial} and \cite[Theorem 9.6]{AKN23}. Statement (2) follows essentially from \cite[Theorem 6.7]{BMS19}, and by the HKR filtration we have $\pi_{2*}\THH(R/\overline{B}_n)\simeq R\langle t'_{1},\ldots,t'_{r}\rangle$, where $R\langle t'_{1},\ldots,t'_{r}\rangle$ is the divided power polynomial on $\{t'_{i}=f_i+\overline{J}^2\in \pi_2\THH(R/\overline{B}_n)=\overline{J}/\overline{J}^2\}_{1\leq i\leq r}$.
\end{proof}

\begin{rem}\label{rem:trivialing-Breuil-Kisin-twists}
    In Corollary \ref{cor:THH-for-qrsp-relative-polynomial}, since the distinguished ideal $I=(d)$ is principal, the Breuil-Kisin twists can be trivialized, and the cyclotomic Frobenius $\varphi_p^{h\T}:\TC^-(R/\B_n)\rightarrow \TP(R/\B_n)$ induces on homotopy groups (up to multiplication by units due to the trivialization) the following composite: 
        $$
        \pi_{2i}\varphi_p^{h\T}: \calN^{\geq i}\widehat{\Prism}^{(1)}_{R/B_n^\wedge}\xrightarrow{\frac{1}{d^i}\cdot-} \Prism_{R/B_n^\wedge}
        \xrightarrow{c}\widehat{\Prism}^{(1)}_{R/B_n^\wedge}.
        $$
\end{rem}

\section{Remarks on Prismatic Envelopes}\label{section:remarks-on-prismatic-envelopes}

In Corollary \ref{cor:THH-for-qrsp-relative-polynomial} we express the prismatic cohomology data in terms of prismatic envelopes (see, for example, \cite[Proposition 3.13]{BS22} and \cite[\S 3]{AKN24}). A prismatic envelope is, in general, quite complicated, but its generators and relations can be recursively determined; see \cite[Proposition 3.8]{AKN24} and also \cite[\S 3]{LW22}. Let us recall some important elements of a prismatic envelope.

\begin{construction}[{\cite[Pages 872-873]{LW22}}, {\cite[\S 3.4]{AKN23}}]\label{construction:generators-for-envelopes}
    Let $(A,I=(d))$ be a bounded orientable prism, and let $f_1,\ldots,f_r\in A$ be elements that form a $p$-completely Koszul regular sequence in $\overline{A}=A/(d)$. We specify the following elements:
    \begin{itemize}
        \item[(a)] Since $A\{\frac{f_1}{d},\ldots,\frac{f_r}{d}\}^\wedge_{(p,d)}$ is a $\delta$-ring, there are elements $\delta^k(\frac{f_i}{d})$ for $1\leq i\leq r$ and $k\geq 0$.

        \item[(b)] For $1\leq i\leq r$, we define the set of elements $\{f_i^{(k)}\}_{k\geq 0}\subset A\{\frac{\phi(f_1)}{\phi(d)},\ldots,\frac{\phi(f_r)}{\phi(d)}\}^\wedge_{(p,\calN)}$ inductively by setting $f_i^{(0)}$ to be the image of $f_i$ under the canonical map $ A\rightarrow A\{\frac{\phi(f_1)}{\phi(d)},\ldots,\frac{\phi(f_r)}{\phi(d)}\}^\wedge_{(p,\calN)}$, and, for $k\geq 0$, 
        $$
        f_i^{(k+1)}=\delta(f_i^{(k)})-\delta^k(\frac{\phi(f_i)}{\phi(d)})\delta(d^{p^k})\in A\{\frac{\phi(f_1)}{\phi(d)},\ldots,\frac{\phi(f_r)}{\phi(d)}\}^\wedge_{(p,\calN)}.
        $$
    \end{itemize}
\end{construction}

\begin{lem}\label{lem:prismatic-envelope-Nygaard-generators}
    Let $(A,I=(d))$ be a bounded orientable prism, and let $f_1,\ldots,f_r\in A$ be elements that form a $p$-completely Koszul regular sequence in $\overline{A}=A/(d)$. Let $\{f_i^{(k)}\}_{k\geq 0}\subset A\{\frac{\phi(f_1)}{\phi(d)},\ldots,\frac{\phi(f_r)}{\phi(d)}\}^\wedge_{(p,\calN)}$ be determined as in Construction \ref{construction:generators-for-envelopes}. Under the inclusion (see \cite[Lemma 3.14]{AKN24}) 
    $$
    A\{\frac{\phi(f_1)}{\phi(d)},\ldots,\frac{\phi(f_r)}{\phi(d)}\}^\wedge_{(p,\calN)}\subset A\{\frac{f_1}{d},\ldots,\frac{f_r}{d}\}^\wedge_{(p,d)},
    $$
    we have 
    $$
    f_i^{(k)}=d^{p^k}\delta^k(\frac{f_i}{d})\in A\{\frac{f_1}{d},\ldots,\frac{f_r}{d}\}^\wedge_{(p,d)},\quad \quad \text{for $k\geq 0$ and $1\leq i\leq r$},
    $$
    and therefore %by definition of the relative Nygaard filtration, we have
    $$
        f_i^{(k)}\in \calN^{\geq p^k}A\{\frac{\phi(f_1)}{\phi(d)},\ldots,\frac{\phi(f_r)}{\phi(d)}\}^\wedge_{(p,\calN)},\quad \text{for $k\geq 0$ and $1\leq i\leq r$}.
    $$
\end{lem}

\begin{proof}
    We argue by induction on $k\geq 0$. When $k=0$, the statement is immediate. Assume that the statement holds for some $k\geq 0$. Then we have 
    \begin{align*}
        f_{i}^{(k+1)} 
        & = \delta(f_i^{(k)})-\delta^k(\frac{\phi(f_i)}{\phi(d)})\delta(d^{p^k}) \\
        & = \delta(d^{p^k}\delta^k(\frac{f_i}{d}))-\phi(\delta^k(\frac{f_i}{d}))\delta(d^{p^k}) \\
        & = \delta(d^{p^k})(\delta^k(\frac{f_i}{d}))^p+d^{p^{k+1}}\delta^{k+1}(\frac{f_i}{d})+p\delta(d^{p^k})\delta^{k+1}(\frac{f_i}{d})
        \\ &\quad -(\delta^k(\frac{f_i}{d}))^p\delta(d^{p^k})-p\delta^{k+1}(\frac{f_i}{d})\delta(d^{p^k}) \\
        & = d^{p^{k+1}}\delta^{k+1}(\frac{f_i}{d}).
    \end{align*}

    By \cite[Lemma 3.14]{AKN24}, we have 
     $$
        f_i^{(k)}\in \calN^{\geq p^k}A\{\frac{\phi(f_1)}{\phi(d)},\ldots,\frac{\phi(f_r)}{\phi(d)}\}^\wedge_{(p,\calN)},\quad k\geq 0,1\leq i\leq r.
    $$
\end{proof}

\begin{prop}\label{prop:Hodge-Tate-and-Nygaard-basis}
    Let $(A,I=(d))$ be an orientable and transversal prism, and let $R$ be an $A$-algebra of the form $A/(d,f_1,\ldots,f_r)$, where the images of $f_1,\ldots,f_r$ form a $p$-completely Koszul regular sequence in $\overline{A}=A/(d)$. Let $\{f_i^{(k)}\}_{k\geq 0}\subset A\{\frac{\phi(f_1)}{\phi(d)},\ldots,\frac{\phi(f_r)}{\phi(d)}\}^\wedge_{(p,\calN)}$ be determined as in Construction \ref{construction:generators-for-envelopes}. Then the following statements hold:
    \begin{itemize}
        \item[(1)] The $R$-module $\Fil^\cj_{\leq l} A\{\frac{f_1}{d},\ldots,\frac{f_r}{d}\}^\wedge_{(p,d)}/d$ admits the following basis: 
        $$
        \{\prod_{j,i}\delta^j(\frac{f_i}{d})^{e_{j,i}}\}_{1\leq i\leq r,j\geq 0,0\leq e_{j,i}<p,\sum_{j,i}e_{j,i}p^j\leq l}.
        $$

        \item[(2)] The $R$-module $\gr^l_\calN A\{\frac{\phi(f_1)}{\phi(d)},\ldots,\frac{\phi(f_r)}{\phi(d)}\}^\wedge_{(p,\calN)}$ admits the following basis: 
        $$
        \{d^{l-\sum_{j,i}e_{j,i}p^j}\prod_{j,i}(f_i^{(j)})^{e_{j,i}}\}_{1\leq i\leq r,j\geq 0,0\leq e_{j,i}<p,\sum_{j,i}e_{j,i}p^j\leq l}.
        $$
    \end{itemize}
\end{prop}

\begin{proof}
    Statement (1) follows from \cite[Proposition 3.9]{AKN23}. 
    
    Note that, by \cite[Lemma 3.14]{AKN24}, we have 
    $$
    \calN^{\geq l}A\{\frac{\phi(f_1)}{\phi(d)},\ldots,\frac{\phi(f_r)}{\phi(d)}\}^\wedge_{(p,\calN)}=A\{\frac{\phi(f_1)}{\phi(d)},\ldots,\frac{\phi(f_r)}{\phi(d)}\}^\wedge_{(p,\calN)}\cap d^l A\{\frac{f_1}{d},\ldots,\frac{f_r}{d}\}^\wedge_{(p,d)}.
    $$
    Moreover, by a suitabe trivialization of the Breuil-Kisin twists, the following equivalence of Notation \ref{notation:relative-prismatic-cohomology} (f)
    $$
    \gr^l_\calN\widehat{\Prism}_{R/A}^{(1)}\{i\}\xrightarrow{\sim} \Fil_{\leq l}^\cj \overline{\Prism}_{R/A}\{l\}
    $$
    can be induced by 
    $$
    -\cdot\frac{1}{d^l}:\calN^{\geq l}A\{\frac{\phi(f_1)}{\phi(d)},\ldots,\frac{\phi(f_r)}{\phi(d)}\}^\wedge_{(p,\calN)}\rightarrow A\{\frac{f_1}{d},\ldots,\frac{f_r}{d}\}^\wedge_{(p,d)}.
    $$
    Therefore, statement (2) follows from Lemma \ref{lem:prismatic-envelope-Nygaard-generators} and statement (1). 
\end{proof}

Construction \ref{construction:generators-for-envelopes} amounts to applying the following map iteratively.

\begin{construction}\label{construction:Nygaard-delta}
    Let $(A,I=(d))$ be a bounded orientable prism, and let $f_1,\ldots,f_r\in A$ be elements that form a $p$-completely Koszul regular sequence in $\overline{A}=A/(d)$. By \cite[Lemma 3.14]{AKN24}, for an integer $l\geq 0$, we have 
    $$
    \calN^{\geq l}A\{\frac{\phi(f_1)}{\phi(d)},\ldots,\frac{\phi(f_r)}{\phi(d)}\}^\wedge_{(p,\calN)}=A\{\frac{\phi(f_1)}{\phi(d)},\ldots,\frac{\phi(f_r)}{\phi(d)}\}^\wedge_{(p,\calN)}\cap d^l A\{\frac{f_1}{d},\ldots,\frac{f_r}{d}\}^\wedge_{(p,d)}.
    $$
    Let $f\in \calN^{\geq l}A\{\frac{\phi(f_1)}{\phi(d)},\ldots,\frac{\phi(f_r)}{\phi(d)}\}^\wedge_{(p,\calN)}$. We set 
    $$
    f^{(1)}:=d^{pl}\delta(\frac{f}{d^l})\in A\{\frac{f_1}{d},\ldots,\frac{f_r}{d}\}^\wedge_{(p,d)}.
    $$
    Note that, we have 
    \begin{align*}
        f^{(1)}
        & = d^{pl}\delta(\frac{f}{d^l}) \\
        & = \delta(f)-(\frac{f}{d^l})^p\delta(d^{l})-p\delta(\frac{f}{d^l})\delta(d^{l}) \\
        & = \delta(f)-\phi(\frac{f}{d^l})\delta(d^{l})\in A\{\frac{\phi(f_1)}{\phi(d)},\ldots,\frac{\phi(f_r)}{\phi(d)}\}^\wedge_{(p,\calN)}.
    \end{align*}
    Therefore, $f^{(1)}\in \calN^{\geq pl}A\{\frac{\phi(f_1)}{\phi(d)},\ldots,\frac{\phi(f_r)}{\phi(d)}\}^\wedge_{(p,\calN)}$. Thus, for each $l\geq 0$, we have the map 
    $$
    (-)^{(1)}:\calN^{\geq l}A\{\frac{\phi(f_1)}{\phi(d)},\ldots,\frac{\phi(f_r)}{\phi(d)}\}^\wedge_{(p,\calN)}\rightarrow \calN^{\geq pl}A\{\frac{\phi(f_1)}{\phi(d)},\ldots,\frac{\phi(f_r)}{\phi(d)}\}^\wedge_{(p,\calN)}.
    $$
\end{construction}

\begin{lem}[{\cite[Proposition 3.23]{LW22}}]\label{lem:Nygaard-generators-divided-power}
    Let $(A,I=(d))$ be a bounded orientable prism, and let $f_1,\ldots,f_r\in A$ be elements that form a $p$-completely Koszul regular sequence in $\overline{A}=A/(d)$. Then the following statement hold:
    \begin{itemize}
        \item Let $\{f_i^{(k)}\}_{k\geq 0}\subset A\{\frac{\phi(f_1)}{\phi(d)},\ldots,\frac{\phi(f_r)}{\phi(d)}\}^\wedge_{(p,\calN)}$ be determined as in Construction \ref{construction:generators-for-envelopes}.  Then 
    $$
    pf_i^{(k+1)}=-(f_i^{(k)})^p+\delta^k(\frac{\phi(f_i)}{\phi(d)})d^{p^{k+1}}, \quad \text{for $k\geq 0$ and $1\leq i\leq r$}.
    $$

    %    \item[(2)] Let $l\geq 1$. For an element $f\in \calN^{\geq l}A\{\frac{\phi(f_1)}{\phi(d)},\ldots,\frac{\phi(f_r)}{\phi(d)}\}^\wedge_{(p,\calN)}$, consider the element $f^{(1)}\in \calN^{\geq pl}A\{\frac{\phi(f_1)}{\phi(d)},\ldots,\frac{\phi(f_r)}{\phi(d)}\}^\wedge_{(p,\calN)}$ determined as in Construction \ref{construction:Nygaard-delta}. Then 
    %    $$
    %    pf^{(1)}= -f^p+\phi(\frac{f}{d^l})d^{pl}.
    %    $$
    \end{itemize}
\end{lem}

\begin{proof}
    We have 
        \begin{align*}
            pf_i^{(k+1)} & =p\delta(f_i^{(k)})-p\delta^k(\frac{\phi(f_i)}{\phi(d)})\delta(d^{p^k}) \\
            & = (\phi(f_i^{(k)})-(f_i^{(k)})^p)-\delta^k(\frac{\phi(f_i)}{\phi(d)})(\phi(d^{p^k})-d^{p^{k+1}}) \\
            & = -(f_i^{(k)})^p+\delta^k(\frac{\phi(f_i)}{\phi(d)})d^{p^{k+1}}+\phi(f_i^{(k)})-\delta^k(\frac{\phi(f_i)}{\phi(d)})\phi(d^{p^k}), 
        \end{align*}
    and by Lemma \ref{lem:prismatic-envelope-Nygaard-generators}, we have $\phi(f_i^{(k)})-\delta^k(\frac{\phi(f_i)}{\phi(d)})\phi(d^{p^k})=0$.

    %(2) We have 
    %    \begin{align*}
    %        pf^{(1)} 
    %        & = p\delta(f)-\phi(\frac{f}{d^l})p\delta(d^{l}) \\
    %        & = (\phi(f)-f^p)-\phi(\frac{f}{d^l})(\phi(d^{l})-d^{pl}) \\
    %        & = -f^p+\phi(\frac{f}{d^l})d^{pl}.
    %    \end{align*}
\end{proof}

\begin{lem}\label{lem:transversal-envelope}
    Let $(A,I=(d))$ be a bounded orientable prism, and let $R$ be an $A$-algebra of the form $A/(d,f_1,\ldots,f_r)$, where the images of $f_1,\ldots,f_r$ form a $p$-completely Koszul regular sequence in $\overline{A}=A/(d)$. Assume further that $R$ is $p$-torsion-free. Then $A\{\frac{f_1}{d},\ldots,\frac{f_r}{d}\}^\wedge_{(p,d)}/d$ is $p$-torsion-free; in other words, $(A\{\frac{f_1}{d},\ldots,\frac{f_r}{d}\}^\wedge_{(p,d)},dA\{\frac{f_1}{d},\ldots,\frac{f_r}{d}\}^\wedge_{(p,d)})$ is a transversal prism (in the sense of \cite[Definition 2.1.3]{BL22}).
\end{lem}

\begin{proof}
    By \cite[Example 7.9]{BS22}, there is a canonical map 
    $$R\rightarrow \overline{\Prism}_{R/A}\simeq A\{\frac{f_1}{d},\ldots,\frac{f_r}{d}\}^\wedge_{(p,d)}/d,
    $$ 
    and by examining the conjugate filtration on $\overline{\Prism}_{R/A}$, we find that $R\rightarrow \overline{\Prism}_{R/A}$ is $p$-completely faithfully flat, and therefore $\overline{\Prism}_{R/A}\simeq A\{\frac{f_1}{d},\ldots,\frac{f_r}{d}\}^\wedge_{(p,d)}/d$ is $p$-torsion-free.     
\end{proof}

\begin{lem}\label{lem:delta-power-Nygaard}
    Let $(A,I=(d))$ be a bounded orientable prism, and let $R$ be an $A$-algebra of the form $A/(d,f_1,\ldots,f_r)$, where the images of $f_1,\ldots,f_r$ form a $p$-completely Koszul regular sequence in $\overline{A}=A/(d)$. Assume further that $R$ is $p$-torsion-free. Let $f\in A$ be an element with $\phi^k(f)\in (d,f_1,\ldots,f_r)$ for all $k\geq 0$. Then 
    $$
    \delta^k(\frac{\phi^j(f)}{\phi^j(d)})\in \calN^{\geq 1}A\{\frac{\phi(f_1)}{\phi(d)},\ldots,\frac{\phi(f_r)}{\phi(d)}\}^\wedge_{(p,\calN)},\quad  \text{for $j\geq 1$ and $k\geq 0$}.
    $$
\end{lem}

\begin{proof}
    Consider the inclusion $A\{\frac{\phi(f_1)}{\phi(d)},\ldots,\frac{\phi(f_r)}{\phi(d)}\}^\wedge_{(p,\calN)}\subset A\{\frac{f_1}{d},\ldots,\frac{f_r}{d}\}^\wedge_{(p,d)}$. We argue by induction on $k\geq 0$. 
    
    When $k=0$, we have $f\in dA\{\frac{f_1}{d},\ldots,\frac{f_r}{d}\}^\wedge_{(p,d)}$, and therefore $\phi^j(f)\in \phi^j(d)A\{\frac{f_1}{d},\ldots,\frac{f_r}{d}\}^\wedge_{(p,d)}$, and, by our assumption, we have $\phi^j(f)\in dA\{\frac{f_1}{d},\ldots,\frac{f_r}{d}\}^\wedge_{(p,d)}\cap \phi^j(d)A\{\frac{f_1}{d},\ldots,\frac{f_r}{d}\}^\wedge_{(p,d)}$. By Lemma \ref{lem:transversal-envelope} the prism $(A\{\frac{f_1}{d},\ldots,\frac{f_r}{d}\}^\wedge_{(p,d)},dA\{\frac{f_1}{d},\ldots,\frac{f_r}{d}\}^\wedge_{(p,d)})$ is transversal; by \cite[Corollary 2.2.9]{BL22}, we have $\phi^j(f)\in d\phi^j(d)A\{\frac{f_1}{d},\ldots,\frac{f_r}{d}\}^\wedge_{(p,d)}$, and therefore $\frac{\phi^j(f)}{\phi^j(d)}\in \calN^{\geq 1}A\{\frac{\phi(f_1)}{\phi(d)},\ldots,\frac{\phi(f_r)}{\phi(d)}\}^\wedge_{(p,\calN)}$.

    Assume that the statement holds for some $k\geq 0$. By Lemma \ref{lem:transversal-envelope}, it suffices to check that $p\delta^{k+1}(\frac{\phi^j(f)}{\phi^j(d)})\in \calN^{\geq 1}A\{\frac{\phi(f_1)}{\phi(d)},\ldots,\frac{\phi(f_r)}{\phi(d)}\}^\wedge_{(p,\calN)}$ for $j\geq 1$. By the inductive hypothesis, we have
    $$
    p\delta^{k+1}(\frac{\phi^j(f)}{\phi^j(d)})=\delta^{k}(\frac{\phi^{j+1}(f)}{\phi^{j+1}(d)})-(\delta^k(\frac{\phi^{j}(f)}{\phi^j(d)}))^p\in \calN^{\geq 1}A\{\frac{\phi(f_1)}{\phi(d)},\ldots,\frac{\phi(f_r)}{\phi(d)}\}^\wedge_{(p,\calN)}.
    $$
\end{proof}

\begin{prop}\label{prop:Nygaard-graded-and-Hodge-Tate}
    Let $(A,I=(d))$ be an orientable and transversal prism, and let $R$ be an $A$-algebra of the form $A/(d,f_1,\ldots,f_r)$, where the images of $f_1,\ldots,f_r$ form a $p$-completely Koszul regular sequence in $\overline{A}=A/(d)$. Assume that at least one of the following conditions is satisfied: 
    \begin{itemize}
        \item[(i)] $R$ is $p$-torsion-free and $\phi^k(f_i)\in (d,f_1,\ldots,f_r)$, for all $k\geq 0$ and $1\leq i\leq r$.

        \item[(ii)] $f_i=g_{i,1}-g_{i,0}$ with $\delta(g_{i,0})=\delta(g_{i,1})=0$ for $1\leq i\leq r$.
    \end{itemize}
    Then the following statements hold:
    \begin{itemize}
        \item[(1)] The ring $A\{\frac{f_1}{d},\ldots,\frac{f_r}{d}\}^\wedge_{(p,d)}/d$ admits a divided power structure on the ideal generated by the classes of elements of $\{\delta^k(\frac{f_i}{d})\}_{k\geq 0,1\leq i\leq r}$, such that we have the following isomorphism  
        $$
        R\langle s_{1},\ldots,s_r\rangle^\wedge_p\xrightarrow[s_i^{[p^k]}\mapsto \delta^k(\frac{f_i}{d}),\ \text{up to multiplication by units}]{\sim}A\{\frac{f_1}{d},\ldots,\frac{f_r}{d}\}^\wedge_{(p,d)}/d,
        $$ 
        where $R\langle s_{1},\ldots,s_r\rangle$ is the divided power polynomial algebra in the variables $s_1,\ldots,s_r$. 

        \item[(2)] The graded ring $\gr^*_\calN A\{\frac{\phi(f_1)}{\phi(d)},\ldots,\frac{\phi(f_r)}{\phi(d)}\}^\wedge_{(p,\calN)}$ admits a graded divided power structure on the ideal
        generated by the classes of elements of $\{f_i^{(k)}\}_{k\geq 0,1\leq i\leq r}$, such that we have the following isomorphism of graded rings 
        $$
        \overline{A}[u]\otimes_{\overline{A}}R\langle t_{1},\ldots,t_r\rangle\xrightarrow[u\mapsto d,t_i^{[p^k]}\mapsto f_i^{(k)},\ \text{up to multiplication by units}]{\sim}\gr^*_\calN A\{\frac{\phi(f_1)}{\phi(d)},\ldots,\frac{\phi(f_r)}{\phi(d)}\}^\wedge_{(p,\calN)},
        $$
        where $\overline{A}[u]$ is a polynomial ring and $R\langle t_{1},\ldots,t_r\rangle$ is a divided power polynomial algebra.
    \end{itemize}
\end{prop}

\begin{proof}
    First, we assume that condition (i) is satisfied. Since $R$ is now $p$-torsion-free, the existence of divided powers is merely a ring-theoretic property.
    
    (1) By \cite[Proposition 3.9]{AKN23}, the $R$-module $\Fil^\cj_{\leq n} A\{\frac{f_1}{d},\ldots,\frac{f_r}{d}\}^\wedge_{(p,d)}/d$ admits a basis of the form 
    $$
    \{\prod_{j,i}\delta^j(\frac{f_i}{d})^{e_{j,i}}\}_{1\leq i\leq r,j\geq 0,0\leq e_{j,i}<p,\sum_{j,i}e_{j,i}p^j\leq n}.
    $$
    On the other hand, by Lemma \ref{lem:delta-power-Nygaard} %(or Lemma \ref{lem:delta-zero-Nygaard}) 
    we have 
    $$
    p\delta^{k+1}(\frac{f_i}{d})+\delta^k(\frac{f_i}{d})^p=\delta^k(\frac{\phi(f_i)}{\phi(d)})\in dA\{\frac{f_1}{d},\ldots,\frac{f_r}{d}\}^\wedge_{(p,d)}, \quad \text{for $k\geq 0$ and $1\leq i\leq r$}.
    $$
    Then there is a ring homomorphism 
    $$
    R\langle s_{1},\ldots,s_r\rangle^\wedge_p\rightarrow A\{\frac{f_1}{d},\ldots,\frac{f_r}{d}\}^\wedge_{(p,d)}/d
    $$ 
    which sends $s_i^{[p^k]}$ to the class of $\delta^k(\frac{f_i}{d})$ up to multiplication by a unit, and induces isomorphisms on the conjugate filtrations. Hence, it is an isomorphism. %By, notation \ref{notation:relative-prismatic-cohomology} (f), \cite[Lemma 3.14]{AKN24}, and Lemma \ref{lem:Nygaard-phi-stable}, for any element $f$ of the ideal of $A\{\frac{f_1}{d},\ldots,\frac{f_r}{d}\}^\wedge_{(p,d)}/d$ generated by the classes of elements of $\{\delta^k(\frac{f_i}{d})\}_{k\geq 0,1\leq i\leq r}$, we have 

    (2) By Notation \ref{notation:relative-prismatic-cohomology} (f), or equivalently Corollary \ref{cor:THH-for-qrsp-relative-polynomial} (2), we have %an exact sequence 
    \begin{align*}
    \ &\cofib( gr^{*-1}_\calN A\{\frac{\phi(f_1)}{\phi(d)},\ldots,\frac{\phi(f_r)}{\phi(d)}\}^\wedge_{(p,\calN)}\xrightarrow{u_d}gr^*_\calN A\{\frac{\phi(f_1)}{\phi(d)},\ldots,\frac{\phi(f_r)}{\phi(d)}\}^\wedge_{(p,\calN)})\\
    & \simeq \gr^\cj_*A\{\frac{f_1}{d},\ldots,\frac{f_r}{d}\}^\wedge_{(p,d)}/d \\
    & \simeq R\langle t'_{1},\ldots,t'_{r}\rangle,
    \end{align*}
    where $R\langle t'_{1},\ldots,t'_{r}\rangle$ is the divided power polynomial algebra on the classes 
    $$
    t'_{i}=f_i+\overline{J}^2\in \overline{J}/\overline{J}^2, \quad \quad 1\leq i\leq r,
    $$ 
    and $\overline{J}=(d,f_1,\ldots,f_r)/(d)$. On the other hand, let $\{f_i^{(k)}\}_{k\geq 0}\subset A\{\frac{\phi(f_1)}{\phi(d)},\ldots,\frac{\phi(f_r)}{\phi(d)}\}^\wedge_{(p,\calN)}$ be determined as in Construction \ref{construction:generators-for-envelopes}; then by Lemma \ref{lem:Nygaard-generators-divided-power} and Lemma \ref{lem:delta-power-Nygaard} %(or Lemma \ref{lem:delta-zero-Nygaard}) 
    we have 
    $$
    pf_{i}^{(k+1)}+(f_i^{(k)})^p=\delta^k(\frac{\phi(f_i)}{\phi(d)})d^{p^{k+1}}\in \calN^{\geq p^{k+1}+1}A\{\frac{\phi(f_1)}{\phi(d)},\ldots,\frac{\phi(f_r)}{\phi(d)}\}^\wedge_{(p,\calN)},\quad \text{for $k\geq 0$ and $1\leq i\leq r$}.
    $$
    Then there is a ring homomorphism 
    $$
    \overline{A}[u_d]\otimes_{\overline{A}}R\langle t_{1},\ldots,t_r\rangle\rightarrow  \gr^*_\calN A\{\frac{\phi(f_1)}{\phi(d)},\ldots,\frac{\phi(f_r)}{\phi(d)}\}^\wedge_{(p,\calN)}
    $$
    which sends $u_d$ to the class of $d$ and $t_i^{[p^k]}$ to the class of $f_i^{(k)}$ up to multiplication by a unit, and induces an isomorphism on the $u_d$-Bocksteins. Hence, it is an isomorphism. 

    Now we treat the case in which condition (ii) is satisfied. Consider the prism $(A[\{x_{i,0},x_{i,1}\}_{1\leq i\leq r}]^\wedge_{(p,I)},IA[\{x_{i,0},x_{i,1}\}_{1\leq i\leq r}]^\wedge_{(p,I)})$, where the $\delta$-structure on $A[\{x_{i,0},x_{i,1}\}_{1\leq i\leq r}]$ is determined by the $\delta$-structure of $A$ and $\delta(x_{i,0})=\delta(x_{i,1})=0$ for $1\leq i\leq r$. We have $A[\{x_{i,0},x_{i,1}\}_{1\leq i\leq r}]^\wedge_p/(d,x_{1,1}-x_{1,0},\ldots,x_{r,1}-x_{r,0})\simeq \overline{A}[x_1,\ldots,x_r]^\wedge_p$, which is $p$-torsion-free by assumption; note that we have $\phi^k(x_{i,1}-x_{i,0})\in (x_{i,1}-x_{i,0})$ for all $k\geq 0$. By Lemma \ref{lem:delta-power-Nygaard}, for $j\geq 1$, $1\leq i'\leq n$, and $k\geq 0$,
    we have 
    $$
    \delta^k(\frac{\phi^j(x_{i',1}-x_{i',0})}{\phi^j(d)})\in \calN^{\geq 1}A[\{x_{i,0},x_{i,1}\}_{1\leq i\leq r}]\{\frac{\phi(x_{i,1}-x_{i,0})}{\phi(d)};1\leq i\leq r\}^\wedge_{(p,\calN)}.
    $$
    Since $\delta(g_{i,0})=\delta(g_{i,1})=0$, the $A$-algebra homomorphism $A[\{x_{i,0},x_{i,1}\}_{1\leq i\leq r}]\xrightarrow{x_{i,0}\mapsto g_{i,0},x_{i,1}\mapsto g_{i,1}}A$ is a morphism of $\delta$-rings. Now we consider the following commutative diagram 
    $$
    \xymatrix@R=50pt@C=75pt{
    A[\{x_{i,0},x_{i,1}\}_{1\leq i\leq r}]^\wedge_{(p,I)} \ar[r]^-{x_{i,1},x_{i,0}\mapsto x_i,\ d\mapsto 0} \ar[d]^{x_{i,0}\mapsto g_{i,0},\ x_{i,1}\mapsto g_{i,1}} & \overline{A}[x_1,\ldots,x_r]^\wedge_p \ar[d]^{x_i\mapsto g_{i,0}} \\
    A \ar[r] & R,
    }
    $$
    which induces a morphism
    $$
    A[\{x_{i,0},x_{i,1}\}_{1\leq i\leq r}]\{\frac{x_{i,1}-x_{i,0}}{d};1\leq i\leq r\}^\wedge_{(p,I)}\xrightarrow{\delta^k(\frac{x_{i,1}-x_{i,0}}{d})\mapsto \delta^k(\frac{f_i}{d})} A\{\frac{f_1}{d},\ldots,\frac{f_r}{d}\}^\wedge_{(p,I)},
    $$
    and a $\calN$-filtered morphism
    $$
    A[\{x_{i,0},x_{i,1}\}_{1\leq i\leq r}]\{\frac{\phi(x_{i,1}-x_{i,0})}{\phi(d)};1\leq i\leq r\}^\wedge_{(p,\calN)}\xrightarrow{(x_{i,1}-x_{i,0})^{(k)}\mapsto f_i^{(k)}} A\{\frac{\phi(f_1)}{\phi(d)},\ldots,\frac{\phi(f_r)}{\phi(d)}\}^\wedge_{(p,\calN)}.
    $$ 
    Thus, we may simply appeal to the case in which condition (i) is satisfied.
\end{proof}

\begin{rem}\label{rem:divided-power-structure}
    Note that in Proposition \ref{prop:Nygaard-graded-and-Hodge-Tate}, the identifications with divided power structures are not canonical and depend on the choice of the elements $d,f_1,\ldots,f_r$. However, let $(A,(d))$ and $(A',(d'))$ be orientable and transversal prisms, and let $R=A/(d,f_1,\ldots,f_r)$ and $R'=A'/(d',f_1',\ldots,f_{r'}')$ be algebras satisfying the assumptions of Proposition \ref{prop:Nygaard-graded-and-Hodge-Tate}. Suppose that $\sigma:A\rightarrow A'$ is a morphism of prisms such that $\sigma(f_i)\in (f_1',\ldots,f_{r'}')A'$ for $1\leq i\leq r$. Then the map $\sigma$ is compatible with the divided power structures described in Proposition \ref{prop:Nygaard-graded-and-Hodge-Tate}.
\end{rem}

\begin{cor}[Compare {\cite[Proposition 3.7]{KN22}}]\label{cor:THH-delta-power-Nygaard}
    Let $k$ be a perfect ring of characteristic $p$, and let $\Sphere_{W(k)}$ be the corresponding ring of spherical Witt vectors with coefficients in $k$. Set 
    $$
    \B_n:=\Sphere_{W(k)}[x_0,x_1,\ldots,x_n]=\Sphere_{W(k)}\otimes (\Sigma^{\infty}_+\N)^{\otimes (n+1)};
    $$    
    We also set $B_n=W(k)[x_0,x_1,\ldots,x_n]$ equipped with the $\delta$-structure determined by the canonical $\delta$-structure on $W(k)$ and by $\delta(x_i)=0$ for $0\leq i\leq n$. Let $I=(d)\subset B_n$ be a distinguished principal ideal and let $R=B_{n}/(I,f_1,\ldots,f_r)$ where $f_1,\ldots,f_r\in B_n$ form a Koszul regular sequence in $\overline{B}_n:=B_{n}/I$. Assume that at least one of the following conditions is satisfied: 
    \begin{itemize}
        \item[(i)] $R$ is $p$-torsion-free and $\phi^k(f_i)\in (d,f_1,\ldots,f_r)$ for $k\geq 0$ and $1\leq i\leq r$.

        \item[(ii)] $\phi(f_i)=g_{i,1}-g_{i,0}$ with $\delta(g_{i,0})=\delta(g_{i,1})=0$ for $1\leq i\leq r$.
    \end{itemize} 
    Then there is an isomorphism of graded rings with divided power strucutres
    $$
    \pi_{2*}\THH(R/\B_n)\simeq \overline{B}_n[u]\otimes_{\overline{B}_n}R\langle t_{1},\ldots,t_r\rangle,
    $$
    where $u$ corresponds to the class of $d\in \calN^{\geq 1}B^\wedge_n\{\frac{\phi(f_1)}{\phi(d)},\ldots,\frac{\phi(f_r)}{\phi(d)}\}^\wedge_{(p,\calN)}$, and $t_i^{[p^k]}$, up to multiplication by a unit, corresponds to the class of $f_i^{(k)}\in \calN^{\geq p^k}B^\wedge_n\{\frac{\phi(f_1)}{\phi(d)},\ldots,\frac{\phi(f_r)}{\phi(d)}\}^\wedge_{(p,\calN)}$.
\end{cor}

\begin{proof}
    This follows from Theorem \ref{thm:THH-for-qrsp-relative-polynomial} (1) and Proposition \ref{prop:Nygaard-graded-and-Hodge-Tate} (2). 
\end{proof}

%\begin{rem}\label{rem:S1-equivariant-lifts}
%    \cite[Proposition 3.7]{KN22}.
%\end{rem}

\section{A Descent Spectral Sequence for \texorpdfstring{$\THH(\Z_p[x]/(px))^\wedge_p$}{THH(Zp[x]/px)}}\label{section:descent-for-THH-Zpx-mod-px}

In this section, we introduce a spectral sequence to compute $\THH(\Z_p[x]/(px))^\wedge_p$, following the descent technique of \cite[\S 4, \S 5]{LW22}. 

\begin{notation}\label{notation:computation}
    In this section, let $\Sphere_p$ denote the $p$-completed sphere spectrum and let $\Sphere_p[z,x]$ denote the polynomial ring over $\Sphere_p$ in the variables $z,x$, and we regard $\pi_0 \Sphere_p[z,x]=\Z_p[z,x]$ as a $\delta$-ring with $\delta(z)=\delta(x)=0$. Let $R=\Z_p[x]/(px)$, and regard $R$ as an $\Einfty$-$\Sphere_p[z,x]$-algebra via the map $\Z_p[z,x]\xrightarrow{z\mapsto p,x\mapsto x}\Z_p[x]/(px)$.
\end{notation}

\begin{construction}\label{construction:descent}
    Consider the natural map 
    $$
    \THH(R)^\wedge_p\rightarrow \THH(R/\Sphere_p[z,x]).
    $$
    Taking the ($p$-completed) \v{C}ech construction we obtain the following augmented cosimplicial $\Einfty$-algebra 
    $$
    \THH(R)^\wedge_p\rightarrow ((\THH(R/\Sphere_p[z,x]))^{\otimes_{\THH(R)^\wedge_p}(\bullet+1)})^\wedge_p\simeq \THH(R/\Sphere_p[z,x]^{\otimes_{\Sphere_p}(\bullet+1)}). 
    $$
    Note that, for each $n\in \N$, $\THH(R/\Sphere_p[z,x]^{\otimes_{\Sphere_p}(n+1)})$ is $p$-complete by \cite[Lemma 3.6]{KN22}.
\end{construction}

\begin{lem}\label{lem:descent}
    The augmented cosimplicial $\Einfty$-algebra $\THH(R)^\wedge_p\rightarrow \THH(R/\Sphere_p[z,x]^{\otimes_{\Sphere_p}(\bullet+1)})$ gives rise to an equivalence 
    $$
    \THH(R)^\wedge_p\rightarrow \varprojlim_{\bullet\in \Delta}\THH(R/\Sphere_p[z,x]^{\otimes_{\Sphere_p}(\bullet+1)}).
    $$
\end{lem}

\begin{proof}
    Notice that the fiber of $\THH(R)^\wedge_p\rightarrow \THH(R/\Sphere_p[z,x])$ is $0$-connected. It suffices to appeal to \cite[Proposition 2.14]{MNN17}. 
\end{proof}

\begin{construction}[{\cite[Page 883-884]{LW22}}]\label{construction:descent-spectral-sequence}
    By Lemma \ref{lem:descent}, the coskeleton filtration of $\THH(R/\Sphere_p[z,x]^{\otimes_{\Sphere_p}(\bullet+1)})$ gives rise to a multiplicative second quadrant homology type spectral sequence converging to $\pi_*\THH(R)^\wedge_p$:
    $$
    E^1_{i,j}=\pi_j \THH(R/\Sphere_p[z,x]^{\otimes_{\Sphere_p}(-i+1)})\Rightarrow \pi_{i+j}\THH(R)^\wedge_p.
    $$
\end{construction}

\begin{rem}\label{rem:dependence}
    Let $\{E^n_{*,*}\}_{n\geq 1}$ denote the spectral sequence from Construction \ref{construction:descent-spectral-sequence}. It turns out that $\{E^n_{*,*}\}_{n\geq 2}$ depends only on $\THH(R)^\wedge_p$: in fact, $\{E^n_{*,*}\}_{n\geq 2}$ can be identified with the spectral sequence associated with the motivic filtration for $\THH(R)^\wedge_p$ following the idea in \cite[Proposition 7.13]{BMS19}. To see this, one might consider the filtration $\varprojlim_{\bullet\in \Delta}\tau_{\geq \star}\THH(R/\Sphere_p[z,x]^{\otimes_{\Sphere_p}(\bullet+1)})$ on $\THH(R)^\wedge_p$ and appeal to \cite[Proposition 7.13]{BMS19}, \cite[Theorem 1.2 (6)]{AKN23}, and \cite[Theorem 2.3.4]{ABBK26}. 
\end{rem}

\begin{lem}\label{lem:Hopf-algebroid}
    The following statements hold:
    \begin{itemize}
        \item[(1)] There is an isomorphism of graded rings with divided power structures
        $$
        \pi_*\THH(R/\Sphere_p[z,x])\simeq R\langle t_{zx}\rangle[u],
        $$
        where $u$ is the class of $(z-p)\in \calN^{\geq 1}\Z_p[z,x]\{\frac{\phi(zx)}{\phi(z-p)}\}^\wedge_{(p,\calN)}$, and $t_{zx}^{[p^k]}$, up to multiplication by a unit, is the class of $(zx)^{(k)}\in \calN^{\geq p^k}\Z_p[z,x]\{\frac{\phi(zx)}{\phi(z-p)}\}^\wedge_{(p,\calN)}$. Note that $|u|=2$ and $|t_{zx}^{[p^k]}|=2p^k$.
        
        \item[(2)] There is an isomorphism of graded rings 
        \begin{align*}
            \pi_*\THH(R/\Sphere_p[z,x]^{\otimes_{\Sphere_p}2})
            & = \pi_*\THH(R/\Sphere_p[z_0,x_0,z_1,x_1]) \\
            & =  R\langle t_{z_1-z_0}, t_{x_1-x_0},t_{z_0x_0}\rangle[u_0],
        \end{align*}
        where 
        %$u,t_{z_0x_0},t_{z_1-z_0},t_{x_1-x_0}$ are the classes of $z_0-p,z_0x_0,z_1-z_0,x_1-x_0\in \calN^{\geq 1}\Z_p[z_0,x_0,z_1,x_1]\{\frac{\phi(z_0x_0)}{\phi(z_0-p)},\frac{\phi(z_1-z_0)}{\phi(z_0-p)},\frac{\phi(x_1-x_0)}{\phi(z_0-p)}\}^\wedge_{(p,\calN)}$ respectively, and 
        $t_{z_0x_0}^{[p^k]},t_{z_1-z_0}^{[p^k]},t_{x_1-x_0}^{[p^k]}$, up to multiplication by a unit, are the classes of $(z_0x_0)^{(k)},(z_1-z_0)^{(k)},(x_1-z_0)^{(k)}\in \calN^{\geq p^k}\Z_p[z_0,x_0,z_1,x_1]\{\frac{\phi(z_0x_0)}{\phi(z_0-p)},\frac{\phi(z_1-z_0)}{\phi(z_0-p)},\frac{\phi(x_1-x_0)}{\phi(z_0-p)}\}^\wedge_{(p,\calN)}$, respectively.

        \item[(3)] The cosimplicial graded ring $\pi_*\THH(R/\Sphere_p[z,x]^{\otimes_{\Sphere_p}(\bullet+1)})$ is a cogroupoid object determined by the graded Hopf algebroid
        $$
        (\pi_*\THH(R/\Sphere_p[z,x]),\pi_*\THH(R/\Sphere_p[z_0,x_0,z_1,x_1])).
        $$

        \item[(4)] Consider the spectral sequence of Construction \ref{construction:descent-spectral-sequence}. Then $\{E^1_{*,*}\}$ can be identified with the cobar complex for $\pi_*\THH(R/\Sphere_p[z,x])$ with respect to the Hopf algebroid
        $$
        (\pi_*\THH(R/\Sphere_p[z,x]),\pi_*\THH(R/\Sphere_p[z_0,x_0,z_1,x_1])),
        $$
        and therefore 
        $$
        E^2_{i,j}\simeq \Ext^{-i,j}_{\pi_*\THH(R/\Sphere_p[z_0,x_0,z_1,x_1])}(\pi_*\THH(R/\Sphere_p[z,x]),\pi_*\THH(R/\Sphere_p[z,x])).
        $$
    \end{itemize}

\end{lem}

\begin{proof}
    Statements (1) and (2) follow immediately from Corollary \ref{cor:THH-delta-power-Nygaard}. 

    For brevity, we set $(A,\Gamma)=(\pi_*\THH(R/\Sphere_p[z,x]),\pi_*\THH(R/\Sphere_p[z_0,x_0,z_1,x_1]))$, and let $\eta_L:A\rightarrow \Gamma$ be the natural map sending $z$ and $x$ to $z_0$ and $x_0$, respectively, and let $\eta_R:A\rightarrow \Gamma$ be the natural map sending $z,x$ to $z_1,x_1$ respectively; then both $\eta_L$ and $\eta_R$ exhibit $\Gamma$ as a free and locally finite graded $A$-module. Statement (3) then follows from the multiplicative property of $\THH$ and the Tor-spectral sequence. Statement (4) follows routinely from (3).  
\end{proof}

\begin{notation}\label{notation:Hopf-algebroid}
    By Lemma \ref{lem:Hopf-algebroid}, the pair of graded rings 
    $$
    (\pi_*\THH(R/\Sphere_p[z,x]),\pi_*\THH(R/\Sphere_p[z_0,x_0,z_1,x_1]))
    $$
    is naturally a Hopf algebroid. For brevity, we set 
    $$
    (A,\Gamma)=(\pi_*\THH(R/\Sphere_p[z,x]),\pi_*\THH(R/\Sphere_p[z_0,x_0,z_1,x_1])).
    $$ 
    We have the following (graded) Hopf algebroid structure maps:
    \begin{itemize}
        \item[(a)] The left unit map $\eta_L:A\rightarrow \Gamma$ sending $z$ and $x$ to $z_0$ and $x_0$, respectively, and the right unit map $\eta_R:A\rightarrow \Gamma$ sending $z$ and $x$ to $z_1$ and $x_1$, respectively. Moreover, both $\eta_L$ and $\eta_R$ exhibit $\Gamma$ as a free and locally finite graded $A$-module.

        \item[(b)] The comultiplication $\Delta:\Gamma\rightarrow \Gamma\otimes_{\eta_R,A,\eta_L}\Gamma$, where 
        $$
        \Gamma\otimes_{\eta_R,A,\eta_L}\Gamma\simeq \pi_*\THH(R/\Sphere_p[z_0,x_0,z_1,x_1,z_2,x_2]),
        $$
        and $\Delta$ sends $z_0,x_0\in \Gamma$ to $z_0,x_0\in \Gamma\otimes_{\eta_R,A,\eta_L}\Gamma$ and sends $z_1,x_1\in \Gamma$ to $z_2,x_2\in \Gamma\otimes_{\eta_R,A,\eta_L}\Gamma$. Note that $\Delta$ is $\eta_L$-$\eta_R$-bilinear.

        \item[(c)] The counit map $\epsilon:\Gamma\rightarrow A$ sending $z_0,z_1$ to $z$ and $x_0,x_1$ to $x$. 

        \item[(d)] The conjugation map $c:\Gamma\rightarrow \Gamma$ sending $z_i,x_i$ to $z_{1-i},x_{1-i}$, $i=0,1$.
    \end{itemize}
    We regard $A$ as a left $\Gamma$-comodule via $\eta_L$.
\end{notation}

\begin{prop}\label{prop:Hopf-algebroid}
    Consider the Hopf algebroid
    $$
    (A,\Gamma)=(\pi_*\THH(R/\Sphere_p[z,x]),\pi_*\THH(R/\Sphere_p[z_0,x_0,z_1,x_1])).
    $$ 
    Recall the following isomorphisms from Lemma \ref{lem:Hopf-algebroid} (1) and (2):
    $$
    \pi_*\THH(R/\Sphere_p[z,x])\simeq R\langle t_{zx}\rangle[u],
    $$
    and 
    $$
    \pi_*\THH(R/\Sphere_p[z_0,x_0,z_1,x_1]) =  R\langle t_{z_1-z_0}, t_{x_1-x_0},t_{z_0x_0}\rangle[u_0].
    $$
    The following statements hold:
    \begin{itemize}
        \item[(1)] The maps $\eta_L, \eta_R,\Delta,\epsilon$ and $c$ are compatible with the divided power structures.

        \item[(2)] $\eta_L(u)=u_0$, $\eta_L(t_{zx})=t_{z_0 x_0}$, $\eta_R(u)=u_0+t_{z_1-z_0}$, and 
        $$
        \eta_R(t_{zx}^{[k]})=\sum_{u,v,w\geq 0,u+v+w=k} (x^ut_{z_1-z_0}^{[u]})\cdot (p^vt_{x_1-x_0}^{[v]})\cdot t_{z_0x_0}^{[w]},\quad \quad k\geq 0.
        $$

        \item[(3)] For $k\geq 0$, $\Delta(t_{z_1-z_0}^{[k]})=\sum_{0\leq j\leq k}t_{z_2-z_1}^{[j]}\cdot t_{z_1-z_0}^{[k-j]}$, and $\Delta(t_{x_1-x_0}^{[k]})=\sum_{0\leq j\leq k}t_{x_2-x_1}^{[j]}\cdot t_{x_1-x_0}^{[k-j]}$. 

        \item[(4)] For $k\geq 0$, $\epsilon(t_{z_1-z_0}^{[k]})=\epsilon(t_{x_1-x_0}^{[k]})=0$.
    \end{itemize}
\end{prop}

\begin{proof}
    Statement (1) follows from the proof of Proposition \ref{prop:Nygaard-graded-and-Hodge-Tate}. The remaining statements are routine.
\end{proof}

Finally, using a chain complex simpler than the cobar complex, we determine $E^2_{i,j}$ as described in Lemma \ref{lem:Hopf-algebroid} (4). 

\begin{construction}\label{construction:resolution}
    For brevity, we set 
    $$
    (A,\Gamma)=(\pi_*\THH(R/\Sphere_p[z,x]),\pi_*\THH(R/\Sphere_p[z_0,x_0,z_1,x_1])),
    $$
    and under the isomorphisms in Lemma \ref{lem:Hopf-algebroid} (1) and (2), we have $\Gamma=A\langle t_{z_1-z_0},t_{x_1-x_0}\rangle$, where we regard $\Gamma$ as an $A$-algebra via $\eta_L:A\rightarrow \Gamma$. Consider the (left) $A\langle t_{x_1-x_0}\rangle$-linear map
    $$
    D_z:\Gamma\xrightarrow{t_{z_1-z_0}^{[k]}\mapsto t_{z_1-z_0}^{[k-1]} } \Gamma ,
    $$
    and the (left) $A\langle t_{z_1-z_0}\rangle$-linear map
    $$
    D_x:\Gamma\xrightarrow{t_{x_1-x_0}^{[k]}\mapsto t_{x_1-x_0}^{[k-1]} } \Gamma .
    $$
    By Proposition \ref{prop:Hopf-algebroid} (3), both $D_z$ and $D_x$ are morphisms of left $\Gamma$-comodules, and $D_x\circ D_z=D_z\circ D_x$. 
    Consider the following complex: 
    $$
    0\rightarrow A\xrightarrow{\eta_L}\Gamma\xrightarrow{(D_z dz,D_xdx)}\Gamma dz\oplus \Gamma dx\xrightarrow{D_x dx \oplus D_z dz} \Gamma (dz\wedge dx)\rightarrow 0,
    $$
    where $dz,dx$, and $dz\wedge dx$, with $|dz|=|dx|=2$ and $|dz\wedge dx|=4$, are technical symbols indicating degree shifts, and by convention $dz\wedge dx=-dx\wedge dz$. It is immediate that this complex is a resolution by relative injectives (see \cite[Definition A1.2.10]{Ravenel04}) of the left $\Gamma$-comodule $A$.

    Consider the left $A$-linear map $\widetilde{D}_z=\epsilon\circ D_z$, that is, 
    \begin{align*}
         \widetilde{D}_z:\Gamma\simeq A\langle t_{z_1-z_0},t_{x_1-x_0}\rangle &\rightarrow A \\
         t_{z_1-z_0} &\mapsto 1, \\
         t_{z_1-z_0}^{[k]}\cdot t_{x_1-x_0}^{[k']} &\mapsto 0,\quad \text{$k\neq 1$ or $k'\neq 0$}.
    \end{align*}
    Let $D'_z:=\widetilde{D}_z\circ \eta_R:A\rightarrow A$; more precisely, under the identification in Lemma \ref{lem:Hopf-algebroid} (1), $D_z'$ is the $R$-linear map 
    \begin{align*}
        D_z':R\langle t_{zx}\rangle[u] & \rightarrow  R\langle t_{zx}\rangle[u]  \\
        t_{zx}^{[k]}u^n& \mapsto nt_{zx}^{[k]}u^{n-1} +xt_{zx}^{[k-1]}u^n.
    \end{align*}
    
    Similarly, consider the left $A$-linear map $\widetilde{D}_x=\epsilon\circ D_x$, that is,
     \begin{align*}
         \widetilde{D}_x:\Gamma\simeq A\langle t_{z_1-z_0},t_{x_1-x_0}\rangle &\rightarrow A \\
         t_{x_1-x_0} &\mapsto 1, \\
         t_{z_1-z_0}^{[k]}\cdot t_{x_1-x_0}^{[k']} &\mapsto 0,\quad \text{$k\neq 0$ or $k'\neq 1$}.
    \end{align*}
    Let $D'_x:=\widetilde{D}_x\circ \eta_R:A\rightarrow A$; more precisely, under the identification in Lemma \ref{lem:Hopf-algebroid} (1), $D_x'$ is the $R[u]$-linear map 
    \begin{align*}
        D_x':R\langle t_{zx}\rangle[u] & \rightarrow  R\langle t_{zx}\rangle[u]  \\
        t_{zx}^{[k]} & \mapsto pt_{zx}^{[k-1]},\quad k\geq 0,
    \end{align*}
    where by convention $t_{zx}^{[-1]}=0$.
\end{construction}

\begin{prop}\label{prop:resolution}
    Consider the Hopf algebroid
    $$
    (A,\Gamma)=(\pi_*\THH(R/\Sphere_p[z,x]),\pi_*\THH(R/\Sphere_p[z_0,x_0,z_1,x_1])).
    $$
    Then $\Ext_\Gamma(A,A)$ is computed by the complex 
    $$
    A\xrightarrow{(D'_zdz,D'_x dx)} A dz\oplus A dx\xrightarrow{D'_xdx \oplus D'_zdz} A (dz\wedge dx),
    $$
    where $D'_z,D'_x:A\rightarrow A$ are defined in Construction \ref{construction:resolution}. More precisely, under the identification in Lemma \ref{lem:Hopf-algebroid} (1), $\Ext_\Gamma(A,A)$ is computed by the complex
    $$
    R\langle t_{zx}\rangle[u]\xrightarrow{(D'_zdz,D'_x dx)} R\langle t_{zx}\rangle[u]dz\oplus R\langle t_{zx}\rangle[u]dx\xrightarrow{D'_xdx\oplus D'_zdz} R\langle t_{zx}\rangle[u](dz\wedge dx),
    $$
    where $D'_z: R\langle t_{zx}\rangle[u]\rightarrow  R\langle t_{zx}\rangle[u]$ is the $R$-linear map sending $t_{zx}^{[k]}u^n$ to $nt_{zx}^{[k]}u^{n-1}+xt_{zx}^{[k-1]}u^n$, and $D'_x: R\langle t_{zx}\rangle[u]\rightarrow  R\langle t_{zx}\rangle[u]$ is the $R[u]$-linear map sending $t_{zx}^{[k]}$ to $pt_{zx}^{[k-1]}$.
\end{prop}

\begin{proof}
    Using the relative injective resolution  
    $$
    0\rightarrow A\xrightarrow{\eta_L}\Gamma\xrightarrow{(D_z dz,D_xdx)}\Gamma dz\oplus \Gamma dx\xrightarrow{D_x dx \oplus D_z dz} \Gamma (dz\wedge dx)\rightarrow 0,
    $$
    $\Ext_\Gamma(A,A)$ is computed by the complex 
    $$
    \Hom_\Gamma(A,\Gamma)\rightarrow \Hom_\Gamma(A,\Gamma)dz\oplus \Hom_\Gamma(A,\Gamma)dx\rightarrow \Hom_\Gamma(A,\Gamma)(dz\wedge dx), 
    $$
    and by \cite[Definition A1.2.1]{Ravenel04}, this complex is identified with 
    $$
    A\xrightarrow{(D'_zdz,D'_x dx)} A dz\oplus A dx\xrightarrow{D'_xdx \oplus D'_zdz} A (dz\wedge dx).
    $$
\end{proof}

\begin{cor}\label{cor:degenerate}
    The spectral sequence in Construction \ref{construction:descent-spectral-sequence} computing $\pi_*\THH(R)^\wedge_p$ collapses at the $E^2$-page.
\end{cor}

\begin{proof}
    Combining Lemma \ref{lem:Hopf-algebroid} (4) and Proposition \ref{prop:resolution}, we see that $E^2_{i,j}=0$ if $i\neq 0,-1,-2$ or $j$ is odd (note that the graded ring $A=R\langle t_{zx}\rangle[u]$ is concentrated in even degrees since $|u|=2$ and $|t_{zx}^{[k]}|=2k$). Hence, the spectral sequence collapses at the $E^2$-page simply for degree reasons.
\end{proof}

\section{A Computation of \texorpdfstring{$\THH(\Z_p[x]/(px))^\wedge_p$}{THH(Zp[x]/px)}}\label{section:computation-of-THH-Zpx-mod-px}

In this final section, we compute the complex appearing in Proposition \ref{prop:resolution}:      
    $$
    E^\bullet:=(R\langle t_{zx}\rangle[u]\xrightarrow{(D'_zdz,D'_x dx)} R\langle t_{zx}\rangle[u]dz\oplus R\langle t_{zx}\rangle[u]dx\xrightarrow{D'_xdx \oplus D'_zdz} R\langle t_{zx}\rangle[u](dz\wedge dx)),
    $$
    where $R=\Z_p[x]/(px)$, and $D'_z: R\langle t_{zx}\rangle[u]\rightarrow  R\langle t_{zx}\rangle[u]$ is the $R$-linear map sending $t_{zx}^{[k]}u^n$ to $nt_{zx}^{[k]}u^{n-1} +xt_{zx}^{[k-1]}u^n $, and $D'_x: R\langle t_{zx}\rangle[u]\rightarrow  R\langle t_{zx}\rangle[u]$ is the $R[u]$-linear map sending $t_{zx}^{[k]}$ to $pt_{zx}^{[k-1]}$. 
    %Then we would be able to compute $\pi_*\THH(\Z_p[x]/(px))^\wedge_p$.

\begin{computation}[$H^0(E^\bullet)$]\label{computation:H0}
    We determine the kernel of 
    $$
    (D_z'dz,D_x'dx):R\langle t_{zx}\rangle[u]\rightarrow R\langle t_{zx}\rangle[u]dz\oplus R\langle t_{zx}\rangle[u]dx.
    $$
    Since this map is homogeneous, it suffices to consider for an integer $l\geq 0$ the map 
    $$
    (D_z'dz,D_x'dx)_l: \oplus_{n+k=l}Rt_{zx}^{[k]}u^{n} \rightarrow (\oplus_{n+k=l-1}Rt_{zx}^{[k]}u^{n} dz)\oplus(\oplus_{n+k=l-1}Ru^{n}t_{zx}^{[k]} dx).
    $$
    Consider an element $\sum_{n,k}f_{n,k}t_{zx}^{[k]}u^n$, where $f_{n,k}\in R=\Z_p[x]/(px)$. By Construction \ref{construction:resolution}, we have 
    \begin{align*}
    (D'_z,D'_x )(\sum f_{n,k}t_{zx}^{[k]}u^n )
    & =(\sum_{n,k}(nf_{n,k}t_{zx}^{[k]}u^{n-1} +xf_{n,k}t_{zx}^{[k-1]}u^n ),\sum_{n,k}f_{n,k}pt_{zx}^{[k-1]}u^n)\\
    & =(\sum_{n,k}((n+1)f_{n+1,k}+xf_{n,k+1})t_{zx}^{[k]}u^n,\sum_{n,k}pf_{n,k+1}t_{zx}^{[k]}u^n),
    \end{align*}
    and $(D'_z,D'_x )(\sum f_{n,k}t_{zx}^{[k]}u^n)=0$ if and only if the following conditions are satisfied: 
    \begin{itemize}
        \item[(i)] $(n+1)f_{n+1,k}+xf_{n,k+1}=0$, for $n,k\geq 0$. 

        \item[(ii)] $pf_{n,k}=0$, for $n\geq 0$ and $k\geq 1$.
    \end{itemize}
    
    We claim the following:
    \begin{claim}\label{claim:H0}
        Suppose that $l=pr+s$ with $r\geq 0$ and $0\leq s\leq p-1$. Then an element $\sum_{n+k=l} f_{n,k}t_{zx}^{[k]}u^n $ lies in the kernel of $(D_z'dz,D_x'dx)_l$ if and only if the following conditions are satisfied:
        \begin{itemize}
            \item[(1)] For $n,k\geq 0$ with $n+k=l$, $f_{n,k}\in xR$.

            \item[(2)] For $n<pr$, $f_{n,l-n}=0$. 

            \item[(3)] Suppose that $n=pr+s'$ with $1\leq s'\leq s$. Then
            $$
            f_{n,l-n}=-\frac{x}{n}f_{n-1,l-n+1}=\frac{x^2}{n(n-1)}f_{n-2,l-n+2}=\cdots=(-1)^{s'}\frac{x^{s'}}{n(n-1)\cdots(pr+1)}f_{pr,s}.
            $$
        \end{itemize}
    \end{claim}
    \begin{proof}[Proof of Claim \ref{claim:H0}]
    The sufficiency of conditions in Claim \ref{claim:H0} is routine. We check the necessity as follows.
   
    We first check the necessity of Claim \ref{claim:H0} (1).
    Condition (ii) requires that $f_{n,k}\in xR$, $n\geq 0,k\geq 1$. By condition (i), we have $nf_{n,0}+xf_{n-1,1}=0$; if $p$ does not divide $n$, then $f_{n,0}\in xR$; if $p\mid n$, then $xf_{n-1,1}\in pR$ and we must have $xf_{n-1,1}=0$ and $nf_{n,0}=0$, and therefore $f_{n,0}\in xR$. 
    
    Then we check the necessity of Claim \ref{claim:H0} (2). For any $n<pr$, we write $n=pr'+s'$ with $0\leq r'<r$ and $0\leq s'\leq p-1$. By Claim \ref{claim:H0} (1) we have $p(r'+1)f_{p(r'+1),l-p(r'+1)}=0$, then condition (i) requires that $xf_{pr'+p-1,l-pr'-p+1}=0$ or equivalently $f_{pr'+p-1,l-pr'-p+1}\in pR$; using Claim \ref{claim:H0} (1) we must have $f_{pr'+p-1,l-pr'-p+1}=0$. Following the same logic, a recursive application of condition (i) implies that $f_{n,l-n}=0$.
  
    Finally, we check the necessity of Claim \ref{claim:H0} (3). If $p$ does not divide $n$, then condition (i) implies that $f_{n,k}=-\frac{x}{n}f_{n-1,k+1}$; therefore, for $n=pr+s'$ where $1\leq s'\leq s$, we have 
    $$
        f_{n,l-n}=-\frac{x}{n}f_{n-1,l-n+1}=\frac{x^2}{n(n-1)}f_{n-2,l-n+2}=\cdots=(-1)^{s'}\frac{x^{s'}}{n(n-1)\cdots(pr+1)}f_{pr,s}.
    $$
    \end{proof}
    
    Now, for $l=pr+s\geq 1$ with $r\geq 0$ and $0\leq s\leq p-1$, we set 
    \begin{align*}
        \epsilon_{l}
        &:=\sum_{n+k=pr+s,pr\leq n\leq pr+s}(-1)^{n-pr}\frac{x^{n-pr}}{n(n-1)\cdots(pr+1)} t_{zx}^{[k]}u^n  \\
        & = \sum_{n+k=l,p\lfloor\frac{l}{p}\rfloor\leq n\leq l}(-1)^{n-p\lfloor\frac{l}{p}\rfloor}\frac{x^{n-p\lfloor\frac{l}{p}\rfloor}}{n(n-1)\cdots(p\lfloor\frac{l}{p}\rfloor+1)} t_{zx}^{[k]}u^n.
    \end{align*}
    Then, by Claim \ref{claim:H0}, we have 
    \begin{align*}
        H^0(E)
        & \simeq R\oplus\bigoplus_{l\geq 1}R\cdot x\epsilon_l \\
        & \simeq R\oplus\bigoplus_{l\geq 1}\F_p[x]x\epsilon_l. 
    \end{align*}

    %If $n=k=0$, then $(D'_z,D'_x )(fu^n\cdot t_{zx}^{[k]})=0$.

    %If $k=0$, $n>0$, then $(D'_z,D'_x )(fu^n\cdot t_{zx}^{[k]})=0$ if and only if $nf=0$, and this is possible only when $f\in xR$ and $p\mid n$.

    %If $k>0$, $n=0$, then $(D'_z,D'_x )(fu^n\cdot t_{zx}^{[k]})=0$ if and only if $xf=pf=0$, that is, if and only if $f=0$. 

    %If $k>0$, $n>0$, then $(D'_z,D'_x )(fu^n\cdot t_{zx}^{[k]})=0$ if and only if $nf=xf=pf=0$, that is, if and only if $f=0$. 

    %Therefore, $H^0(E)\simeq R\oplus_{n\geq 1}xRu^{pn}$, where $|u|=2$. 
\end{computation}

\begin{computation}[$H^2(E^\bullet)$]\label{computation:H2}
    We determine the cokernel of 
    $$
    D_x'dx\oplus D_z'dz:R\langle t_{zx}\rangle[u]dz\oplus R\langle t_{zx}\rangle[u]dx\rightarrow R\langle t_{zx}\rangle[u](dz\wedge dx).
    $$
    Since this map is homogeneous, it suffices, for each integer $l\geq 0$, to consider the map 
    $$
    (D_x'dx\oplus D_z'dz)_{l+1}: (\oplus_{n+k=l+1}Rt_{zx}^{[k]}u^{n} dz)\oplus (\oplus_{n+k=l+1}Rt_{zx}^{[k]}u^{n} dx)\rightarrow \oplus_{n+k=l}Rt_{zx}^{[k]}u^{n} (dz\wedge dx).
    $$
    Note that the subscript of $(D_x'dx\oplus D_z'dz)_{l+1}$ concerns only $t_{zx}^{[k]}u^n $ and does not indicate the correct degree in the graded complex $E^\bullet$.
    
    Consider a pair $(\sum_{n+k=l+1}f_{n,k}t_{zx}^{[k]}u^n,\sum_{n+k=l+1}g_{n,k}t_{zx}^{[k]}u^n)$, where $f_{n,k},g_{n,k}\in R=\Z_p[x]/(px)$. By Construction \ref{construction:resolution}, we have 
    \begin{align*}
   \ & (-D_x'\oplus D_z')(\sum_{n,k}f_{n,k}t_{zx}^{[k]}u^n ,\sum_{n,k}g_{n,k}t_{zx}^{[k]}u^n))\\
    & = -\sum_{n,k}f_{n,k}pt_{zx}^{[k-1]}u^n + \sum_{n,k}(ng_{n,k}t_{zx}^{[k]}u^{n-1} +xg_{n,k}t_{zx}^{[k-1]}u^n) \\
    & = \sum_{n,k}(-pf_{n,k+1}+(n+1)g_{n+1,k}+xg_{n,k+1})t_{zx}^{[k]}u^n ,
    \end{align*}
    and since the $f_{n,k}$ can be arbitrary, it suffices to consider the cokernel of 
    $$
        (D_z'dz)_{l+1}:\oplus_{n+k=l+1}Rt_{zx}^{[k]}u^{n} dx
        \rightarrow \oplus_{n+k=l}\F_p[x]t_{zx}^{[k]}u^{n} (dz\wedge dx).
    $$
    We claim the following:
    \begin{claim}\label{claim:H2}
    Suppose that $l+1=pr+s$ with $r\geq 0$ and $0\leq s\leq p-1$. Then an element $\sum_{n+k=l}h_{n,k}t_{zx}^{[k]}u^n dz\wedge dx$, where $h_{n,k}\in \F_p[x]$, lies in the image of $(D_z'dz)_{l+1}$ if and only if the following condition is satisfied:
    \begin{itemize}
        %\item[(1)] For $n\geq pr$, $h_{n,k}$ can be arbitrary. 

        \item[($\star$)] For $1\leq i\leq r$, 
        \begin{align*}
         h_{pi-1,p(r-i)+s}-\frac{x}{pi-1}h_{pi-2,p(r-i)+s+1}+\frac{x^2}{(pi-1)(pi-2)}h_{pi-3,p(r-i)+s+2}\\
         +\cdots +(-1)^{p-1}\frac{x^{p-1}}{(pi-1)\cdots(p(i-1)+1)}h_{p(i-1),p(r-i+1)+s-1}\in x^p\F_p[x],
        \end{align*}
        or equivalently, in $\F_p[x]/(x^p)$, 
        \begin{align*}
         h_{pi-1,p(r-i)+s}-\frac{x}{pi-1}h_{pi-2,p(r-i)+s+1}+\frac{x^2}{(pi-1)(pi-2)}h_{pi-3,p(r-i)+s+2}\\
         +\cdots +(-1)^{p-1}\frac{x^{p-1}}{(pi-1)\cdots(p(i-1)+1)}h_{p(i-1),p(r-i+1)+s-1}=0.
        \end{align*}
    \end{itemize}
    As a consequence, the following $\F_p[x]$-linear map 
    \begin{align*}
        & \oplus_{n+k=l}\F_p[x] t_{zx}^{[k]}u^{n}(dz\wedge dx) \rightarrow  \oplus_{1\leq i\leq r}(\F_p[x]/(x^p)) t_{zx}^{[p(t-i)+s]}u^{pi-1}(dz\wedge dx) \\
        &\sum_{n+k=l} h_{n,k}t_{zx}^{[k]}u^n(dz\wedge dx)  \mapsto \\ &\sum_{1\leq i\leq r} (\sum_{j=1}^{p}\frac{(-x)^{j-1}}{(pi-1)\cdots(pi-j+1)} h_{pi-j,p(r-i)+s+j-1})t_{zx}^{[p(r-i)+s]}u^{pi-1} (dz\wedge dx),
    \end{align*}
    induces an isomorphism $\coker(D_z' dz)\simeq \oplus_{1\leq i\leq r}(\F_p[x]/(x^p))t_{zx}^{[p(r-i)+s]}u^{pi-1} (dz\wedge dx)$.
    \end{claim}
    \begin{proof}[Proof of Claim \ref{claim:H2}]
        In fact, for an element $\sum_{n+k=l}h_{n,k}u^n\cdot t_{zx}^{[k]}(dz\wedge dx)$ to lie in the image of $D_z'dz$, we need to find $g_{n,k}\in \F_p[x]$ such that $(n+1)g_{n+1,k}+xg_{n,k+1}=h_{n,k}$, and more precisely, 
    \begin{itemize}
        \item[(i)] If $p\mid n+1$, then $xg_{n,k+1}=h_{n,k}$. 

        \item[(ii)] If $p$ does not divide $n+1$, then $g_{n+1,k}=\frac{1}{n+1}h_{n,k}-\frac{x}{n+1}g_{n,k+1}$. 
    \end{itemize}
    It is then immediate that condition ($\star$) in Claim \ref{claim:H2} is necessary. Conversely, suppose that condition ($\star$) is satisfied. Noting that $\F_p[x]$ is $x$-torsion-free, we just take $g_{pi-1,p(t-i)+s+1}=\frac{h_{pi-1,p(t-i)+s}}{x}$, $1\leq i\leq r$, and the remaining $g_{n,k}$'s can then be determined.
    \end{proof}

    Hence, we have 
    \begin{align*}
    H^2(E)
    &\simeq \bigoplus_{l+1=pr+s, r\geq 0,0\leq s\leq p-1}\oplus_{1\leq i\leq r}(\F_p[x]/(x^p))t_{zx}^{[p(r-i)+s]}u^{pi-1} (dz\wedge dx)\\
    &\simeq \bigoplus_{l\geq 0}\oplus_{1\leq i\leq \lfloor\frac{l+1}{p}\rfloor}(\F_p[x]/(x^p))t_{zx}^{[l-pi+1]}u^{pi-1} (dz\wedge dx).
    \end{align*}

%    We claim that, an element $\sum_{n,k}h_{n,k}u^n\cdot t_{zx}^{[k]}\in R\langle t_{zx}\rangle[u]$, where $h_{n,k}\in R$, lies in the image of $-D_x'\oplus D_z'$ if and only if, for $n$ such that $p\mid n+1$, we have $h_{n,k}\in (p,x)R$.

%    The ``only if'' direction is obvious. Now given an element $\sum_{n,k}h_{n,k}u^n\cdot t_{zx}^{[k]}\in R\langle t_{zx}\rangle[u]$, such that $h_{n,k}\in (p,x)R$ if $p\mid n+1$. We would determine a collection of elements $\{f_{n,k},g_{n,k}\}\subset R$ by induction on $n+k$ and $n$, such that 
%    $$
%    (-D_x'\oplus D_z')(\sum_{n,k}f_{n,k}u^n\cdot t_{zx}^{[k]},\sum_{n,k}g_{n,k}u^n\cdot t_{zx}^{[k]}))=\sum_{n,k}h_{n,k}u^n\cdot t_{zx}^{[k]}.
%    $$

%    The elements $f_{n,0},g_{0,0}$ can be arbitrary, and we just take $f_{n,0}=g_{0,0}=0$. We would like to find $f_{0,k+1},g_{1,k},g_{0,k+1}$, such that $-pf_{0,k+1}+g_{1,k}+x g_{0,k+1}=h_{0,k}$; we can just take $f_{0,k+1}=g_{0,k+1}=0$ and $g_{1,k}=h_{0,k}$. Now, suppose that for some $l\geq 0$ and some $0\leq n_0<l$ and $k_0\geq 0$ with $n_0+k_0=l$, we have found $f_{n,k+1},g_{n+1,k},g_{n,k+1}$ where $n+k=l$ and $n\leq n_0$, such that $-p f_{n,k+1}+(n+1)g_{n+1,k}+xg_{n,k+1}=h_{n,k}$; we then would like to find $f_{n_0+1,k_0},g_{n_0+2,k_0-1},g_{n_0+1,k_0}$ such that $-p f_{n_0+1,k_0}+(n_0+2)g_{n_0+2,k_0-1}+xg_{n_0+1,k_0}=h_{n_0 +1,k_0-1}$. 
\end{computation}

\begin{computation}[$\ker(D_x'dx\oplus D_z'dz)$]\label{computation:ker1}
    We determine
    $$
    \ker(D_x'dx\oplus D_z'dz:R\langle t_{zx}\rangle[u]dz\oplus R\langle t_{zx}\rangle[u]dx\rightarrow R\langle t_{zx}\rangle[u](dz\wedge dx)).
    $$
    Since the map $D_x'dx\oplus D_z'dz$ is homogeneous, it suffices, for each integer $l\geq 0$, to consider  
    $$
    \ker((D_x'dx\oplus D_z'dz)_l: (\oplus_{n+k=l} R t_{zx}^{[k]}u^n dz)\oplus (\oplus_{n+k=l}R t_{zx}^{[k]}u^n dx)\rightarrow \oplus_{n+k=l-1}Rt_{zx}^{[k]}u^n (dz\wedge dx)).
    $$
    Note that the subscript of $(D_x'dx\oplus D_z'dz)_l$ concerns only $t_{zx}^{[k]}u^n$ and does not correspond to the correct degree in the graded complex $E^\bullet$.
    
    Consider a pair $(\sum_{n+k=l}f_{n,k}t_{zx}^{[k]}u^n ,\sum_{n+k=l}g_{n,k}t_{zx}^{[k]}u^n )$, where $f_{n,k},g_{n,k}\in R=\Z_p[x]/(px)$. By Construction \ref{construction:resolution}, we have 
    \begin{align*}
   \ & (-D_x'\oplus D_z')(\sum_{n,k}f_{n,k}t_{zx}^{[k]}u^n,\sum_{n,k}g_{n,k}t_{zx}^{[k]}u^n)\\
    & = -\sum_{n,k}f_{n,k}pt_{zx}^{[k-1]}u^n + \sum_{n,k}(ng_{n,k}t_{zx}^{[k]}u^{n-1} +xg_{n,k}t_{zx}^{[k-1]}u^n) \\
    & = \sum_{n,k}(-pf_{n,k+1}+(n+1)g_{n+1,k}+xg_{n,k+1})t_{zx}^{[k]}u^n,
    \end{align*}
    and $(-D_x'\oplus D_z')(\sum_{n,k}f_{n,k}t_{zx}^{[k]}u^n ,\sum_{n,k}g_{n,k}t_{zx}^{[k]}u^n )=0$ if and only if the following condition is satisfied: 
    \begin{itemize}
        \item[($*$)] For $n,k\geq 0$, $-pf_{n,k+1}+(n+1)g_{n+1,k}+xg_{n,k+1}=0$.
    \end{itemize}

    We claim the following:
    \begin{claim}\label{claim:ker1}
        Suppose that $l=pr+s$ with $r\geq 0$ and $0\leq s\leq p-1$. Then $(-D_x'\oplus D_z')_l(\sum_{n,k}f_{n,k}t_{zx}^{[k]}u^n ,\sum_{n,k}g_{n,k}t_{zx}^{[k]}u^n )=0$ if and only if the following conditions are satisfied:
        \begin{itemize}
             \item[(1)] For $1\leq r'\leq r$, 
            $$
            r'g_{pr',l-pr'}-f_{pr'-1,l-pr'+1}\in xR.
            $$
            
            \item[(2)] For $n\leq l$ which is not divisible by $p$, write $pi_n<n<p(i_n+1)$, where $0\leq i_n\leq r$, then 
            $$
            g_{n,l-n}=(\sum_{j=1}^{n-pi_n}\frac{(-x)^{j-1}p}{n(n-1)\cdots(n-j+1)}f_{n-j,l-n+j})+\frac{(-x)^{n-pi_n}}{n(n-1)\cdots(pi_n+1)}g_{pi_n,l-pi_n}.
            $$

            %\item[(3)] The elements $f_{l,0}$ and $g_{0,l}$ can be arbitrary.
        \end{itemize}
    \end{claim}
    \begin{proof}[Proof of Claim \ref{claim:ker1}]
        The sufficiency of the conditions in Claim \ref{claim:ker1} is routine. We check the necessity as follows.

        First, we check the necessity of Claim \ref{claim:ker1} (1). By condition ($*$), if $p\mid n$ and $n\geq 1$, then we have $xg_{n-1,k+1}=pf_{n-1,k+1}-ng_{n,k}\in pR$, and we must have $g_{n-1,k+1}\in pR$; therefore $xg_{n-1,k+1}=0$ and $ng_{n,k}-pf_{n-1,k+1}=0$; in $R$, $ng_{n,k}-pf_{n-1,k+1}=0$ if and only if $\frac{n}{p}g_{n,k}-f_{n-1,k+1}\in xR$. 

        Then, we check the necessity of Claim \ref{claim:ker1} (2). When $n\geq 1$ is not divisible by $p$, by condition ($*$) we have
        $$
        g_{n,k}=\frac{p}{n}f_{n-1,k+1}-\frac{x}{n}g_{n-1,k+1},
        $$
        and it suffices to apply this equality recursively to prove Claim \ref{claim:ker1} (2).
    \end{proof}

    Now, for $0\leq l'< l$ with $l'+1$ not divisible by $p$, we set 
    $$
    \beta_{l,l'}=t_{zx}^{[l-l']}u^{l'} dz+\sum_{n=l'+1}^{\min\{p(\lfloor \frac{l'}{p}\rfloor+1)-1,l\}}\frac{(-x)^{n-l'}p}{n(n-1)\cdots(l'+1)}t_{zx}^{[l-n]}u^n dx.
    $$
    Also, for $1\leq i\leq \lfloor\frac{l}{p}\rfloor$, we set 
    $$
    \gamma_{l,i}=it_{zx}^{[l-pi+1]}u^{pi-1} dz+\sum_{n=pi}^{\min\{p(i+1)-1,l\}}\frac{(-x)^{n-pi}}{n(n-1)\cdots(pi+1)}t_{zx}^{[l-n]}u^n dx,
    $$
    where, by convention, $pi(pi-1)\cdots(pi+1)=1$. Finally, 
    $$
    \gamma_{l,0}=\sum_{n=0}^{\min\{p-1,l\}}\frac{(-x)^{n}}{n(n-1)\cdots 1}t_{zx}^{[l-n]}u^n dx,
    $$
    where, by convention, $0(0-1)\cdots 1=1$. By Claim \ref{claim:ker1}, we have 
    \begin{align*}
    \ker(D_x'dx\oplus D_z'dz)\simeq (R[u]dz\oplus R dx)\oplus(\oplus_{l\geq 1}(\oplus_{\text{$0\leq l'< l$, $p\nmid l'+1$}}R\beta_{l,l'})\\\oplus(\oplus_{1\leq i\leq \lfloor\frac{l}{p}\rfloor}R\gamma_{l,i}+Rxt_{zx}^{[l-pi+1]}u^{pi-1} dz)\oplus R\gamma_{l,0}).
    \end{align*}
\end{computation}

\begin{computation}[$\ima(D_z'dz,D_x'dx)$]\label{computation:im1}
    We determine the image of 
    $$
    (D_z'dz,D_x'dx):R\langle t_{zx}\rangle[u]\rightarrow R\langle t_{zx}\rangle[u]dz\oplus R\langle t_{zx}\rangle[u]dx.
    $$
    Since this map is homogeneous, it suffices, for each integer $l\geq 0$,  to consider the map 
    $$
    (D_z'dz,D_x'dx)_{l+1}: \oplus_{n+k=l+1}Rt_{zx}^{[k]}u^{n}\rightarrow (\oplus_{n+k=l}Rt_{zx}^{[k]}u^{n} dz)\oplus(\oplus_{n+k=l}Rt_{zx}^{[k]}u^{n} dx).
    $$
    Note that the subscript of $(D_z'dz,D_x'dx)_{l+1}$ dependss only $t_{zx}^{[k]}u^n$ and does not correspond to the correct degree in the graded complex $E^\bullet$.    
    
    Consider an element $\sum_{n,k}e_{n,k}t_{zx}^{[k]}u^n$, where $e_{n,k}\in R=\Z_p[x]/(px)$. By Construction \ref{construction:resolution}, we have 
    \begin{align*}
    (D'_z,D'_x )(\sum e_{n,k}t_{zx}^{[k]}u^n)
    & =(\sum_{n,k}(ne_{n,k}t_{zx}^{[k]}u^{n-1} +xe_{n,k}t_{zx}^{[k-1]}u^n ),\sum_{n,k}e_{n,k}pt_{zx}^{[k-1]}u^n )\\
    & =(\sum_{n,k}((n+1)e_{n+1,k}+xe_{n,k+1})t_{zx}^{[k]}u^n,\sum_{n,k}pe_{n,k+1}t_{zx}^{[k]}u^n).
    \end{align*}

    It is obvious that the image of $(D_z'dz,D_x'dx)_1$ is $Rdz\oplus pRdx$. For $l\geq 1$, let $\beta_{l,l'}$, $\gamma_{l,i}$, and let $\gamma_{l,0}$ be determined as in Computation \ref{computation:ker1}. We express an element $\alpha\in \ker(D_x'dx\oplus D_z'dz)_l$ as 
        $$
        \alpha=f_lu^ldz+(\sum_{\text{$0\leq l'<l$, $p\nmid l'+1$}}b_{l'}\beta_{l,l'
        })+(\sum_{i=1}^{\lfloor\frac{l}{p}\rfloor}g_{i}\gamma_{l,i}+g_i'xt_{zx}^{[l-pi+1]}u^{pi-1}dz)+g_0\gamma_{l,0},
        $$
        where $f_l,b_{l'},g_i\in R$, $g_i'\in \F_p[x]$. Then $\alpha$ lies in $\ima(D'_zdz,D'_xdx)_{l+1}$ if and only if the following conditions are satisfied:
        \begin{itemize}
            \item[($*$)] There exist elements $e_{n,k}\in R$ with $n+k=l+1$ such that $f_l=(l+1)e_{l+1,0}+xe_{l,1}$, $g_0=pe_{0,l+1}$; for $1\leq i\leq \lfloor\frac{l}{p}\rfloor$, $g_i=pe_{pi,l-pi+1}$, $g_i'=e_{pi-1,l-pi+2}$; for $0\leq l'< l$ with $p\nmid l'+1$, $b_{l'}=(l'+1)e_{l'+1,l-l'}+xe_{l',l-l'+1}$.
        \end{itemize}

    We claim the following: 
    \begin{claim}\label{claim:im1}
    \ 
    
        \begin{itemize}
            \item[(i)] The image of $(D_z'dz,D_x'dx)_1$ is $Rdz\oplus pRdx$.

            \item[(ii)] Suppose that $l\geq 1$ and let $\alpha\in \ker(D_x'dx\oplus D_z'dz)_l$ be expressed as above. Then $\alpha$ lies in $\ima(D_z'dz,D_x'dx)_{l+1}$ if and only if there exist elements $e_i\in R$ for $0\leq i\leq \lfloor\frac{l}{p}\rfloor$ such that the following conditions are satisfied: 
            \begin{itemize}
            %\item[(1)] For $0\leq l'<l$ with $p\nmid l'+1$, the elements $b_{l'}\in R$ can be arbitrary.
        
                \item[(1)] For $1\leq i\leq \lfloor\frac{l}{p}\rfloor$, 
                $$
                g_i'-(\sum_{j=1}^{p-1}\frac{(-x)^{j-1}}{(pi-1)(pi-2)\cdots(pi-j)}b_{pi-j-1})-\frac{(-x)^{p}}{(pi-1)\cdots(pi-p+1)}e_{i-1}\in (p,x^{p+1})R.
                $$

                \item[(2)] For $0\leq i\leq \lfloor\frac{l}{p}\rfloor$, $g_i=pe_i$.

                \item[(3)] If $p\mid l+1$, then 
                $$
                f_l+(\sum_{j=1}^{p-1}\frac{(-x)^{j}}{l(l-1)\cdots(l-j+1)}b_{l-j})-\frac{(-x)^{p}}{l(l-1)\cdots(l-p+2)}e_{\lfloor\frac{l}{p}\rfloor}\in (l+1,x^{p+1})R.
                $$
            \end{itemize}
        \end{itemize}

            %\item[(ii)] If $p|l+1$, then $\alpha$ lies in $\ima(D_z'dz,D_x'dx)_{l+1}$ if and only if there exist elements $e_i\in R$, $0\leq i\leq \frac{l+1}{p}$ such that the following conditions are satisfied:
            %\begin{itemize}
            %\item[(1)] For $0\leq l'<l$ with $p\nmid l'+1$, the elements $b_{l'}\in R$ can be arbitrary.
        
            %    \item[(1)] For $1\leq i<\frac{l+1}{p}$,
            %    $$
            %    g_i'-(\sum_{j=1}^{p-1}\frac{(-x)^{j-1}}{(pi-1)(pi-2)\cdots(pi-j)}b_{pi-j-1})-\frac{(-x)^{p}}{(pi-1)\cdots(pi-p+1)}e_{i-1}=0.
            %    $$

            %    \item[(2)] For $0\leq i<\frac{l+1}{p}$, $g_i=pe_i$.

            %    \item[(3)] $f_l-(l+1)e_{\frac{l+1}{p}}+(\sum_{j=1}^{p-1}\frac{(-x)^{j}}{l(l-1)\cdots(l-j+1)}b_{l-j})-\frac{(-x)^{p}}{l(l-1)\cdots(l-p+2)}e_{\frac{l+1}{p}-1}=0$.
            %\end{itemize}   
    \end{claim}
    \begin{proof}[Proof of Claim \ref{claim:im1}]
         Indeed, using condition ($*$) it is routine to check Claim \ref{claim:im1}: it suffices to note that, in condition ($*$), any $e_{n,k}$ with $p\nmid n$ can be determined once $b_{l'}$'s and $e_{pi,l-pi+1}$ are given, and the elements $e_{pi,l-pi+1}$ function as the elements $e_i$ in Claim \ref{claim:im1} up to additions by $xR$.
    \end{proof}
\end{computation}

\begin{computation}[$H^1(E^\bullet)$]\label{computation:H1}
    We determine 
    $$
    H^1(E^\bullet)=\ker(D_x'dx\oplus D_z'dz)/\ima(D_z'dz,D_x'dx).
    $$
    By homogeneity, for each integer $l\geq 0$, we consider 
    $$
    \ker(D_x'dx\oplus D_z'dz)_l/\ima(D_z'dz,D_x'dx)_{l+1},
    $$
    where $(D_x'dx\oplus D_z'dz)_l$ and $(D_z'dz,D_x'dx)_{l+1}$ are taken as in Computations \ref{computation:ker1} and \ref{computation:im1}, respectively.

    It is obvious that $\ker(D_x'dx\oplus D_z'dz)_0/\ima(D_z'dz,D_x'dx)_{1}\simeq (Rdz\oplus Rdx)/(Rdz\oplus pRdx)\simeq (R/p)dx=\F_p[x]dx$. 

    Now suppose that $l\geq 1$. Let $\beta_{l,l'}$, $\gamma_{l,i}$, and $\gamma_{l,0}$ be determined as in Computation \ref{computation:ker1}. We have 
    \begin{align*}
        \ker(D_x'dx\oplus D_z'dz)_l=Ru^ldz\oplus(\oplus_{\text{$0\leq l'< l$, $p\nmid l'+1$}}R\beta_{l,l'})\\\oplus(\oplus_{1\leq i\leq \lfloor\frac{l}{p}\rfloor}R\gamma_{l,i}+Rxt_{zx}^{[l-pi+1]}u^{pi-1} dz)\oplus R\gamma_{l,0},
    \end{align*}
    and we express an element $\alpha\in \ker(D_x'dx\oplus D_z'dz)_l$ as 
        $$
        \alpha=f_lu^ldz+(\sum_{\text{$0\leq l'<l$, $p\nmid l'+1$}}b_{l'}\beta_{l,l'
        })+(\sum_{i=1}^{\lfloor\frac{l}{p}\rfloor}g_{i}\gamma_{l,i}+xg_i't_{zx}^{[l-pi+1]}u^{pi-1} dz)+g_0\gamma_{l,0},
        $$
        where $f_l,b_{l'},g_i\in R$ and $g_i'\in \F_p[x]$.

    Also, we consider the following $R$-submodule of $\ker(D'_xdx\oplus D_z'dz)_l$:
    \begin{align*}
         \widetilde{\ker}(D'_xdx\oplus D_z'dz)_l:=Ru^ldz\oplus(\oplus_{\text{$0\leq l'< l$, $p\nmid l'+1$}}R\beta_{l,l'})\\\oplus(\oplus_{1\leq i\leq \lfloor\frac{l}{p}\rfloor}Rp\gamma_{l,i}+Rxt_{zx}^{[l-pi+1]}u^{pi-1} dz)\oplus Rp\gamma_{l,0},
    \end{align*}
    and we express an element $\tilalpha\in \widetilde{\ker}(D_x'dx\oplus D_z'dz)_l$ as 
        $$
        \tilalpha=\tilf_lu^ldz+(\sum_{\text{$0\leq l'<l$, $p\nmid l'+1$}}\tilb_{l'}\beta_{l,l'
        })+(\sum_{i=1}^{\lfloor\frac{l}{p}\rfloor}\tilg_{i}p\gamma_{l,i}+x\tilg_i't_{zx}^{[l-pi+1]}u^{pi-1} dz)+\tilg_0p\gamma_{l,0},
        $$
        where $\tilf_l,\tilb_{l'}\in R$, $\tilg_i\in \Z_p$ and $\tilg_i'\in \F_p[x]$. It is obvious that 
    \begin{align*}
        \ker(D'_xdx\oplus D_z'dz)_l/\widetilde{\ker}(D'_xdx\oplus D_z'dz)_l\simeq \oplus_{0\leq i\leq \lfloor\frac{l}{p}\rfloor}(R/p)\gamma_{l,i} \\
        \simeq \oplus_{0\leq i\leq \lfloor\frac{l}{p}\rfloor}\F_p[x]\gamma_{l,i}.
    \end{align*}

    Now we claim the following:
    \begin{claim}\label{claim:H1}
        Suppose that $l\geq 1$, and let $\widetilde{\ker}(D_x'dx\oplus D_z'dz)_l$ be expressed as above. Then the following statements hold.
        \begin{itemize}
            \item[(1)] We have $\ima(D_z'dz,D_x'dx)_{l+1}\subset \widetilde{\ker}(D_x'dx\oplus D_z'dz)_l$. 

            \item[(2)] Suppose that $p\nmid l+1$, and consider the $R$-linear map 
            $$
            \widetilde{\ker}(D_x'dx\oplus D_z'dz)_l\rightarrow \oplus_{1\leq i\leq \lfloor\frac{l}{p}\rfloor}(\F_p[x]/(x^{p+1}))xt_{zx}^{[l-pi+1]}u^{pi-1} dz,
            $$
            which sends an element $\tilalpha=\tilf_lu^ldz+(\sum_{\text{$0\leq l'<l$, $p\nmid l'+1$}}\tilb_{l'}\beta_{l,l'
            })+(\sum_{i=1}^{\lfloor\frac{l}{p}\rfloor}\tilg_{i}p\gamma_{l,i}+x\tilg_i't_{zx}^{[l-pi+1]}u^{pi-1} dz)+\tilg_0p\gamma_{l,0}$ to $\sum_{i=1}^{\lfloor\frac{l}{p}\rfloor}(\tilg_i'-(\sum_{j=1}^{p-1}\frac{(-x)^{j-1}}{(pi-1)(pi-2)\cdots(pi-j)}\tilb_{pi-j-1})-\frac{(-x)^{p}}{(pi-1)\cdots(pi-p+1)}\tilg_{i-1})xt_{zx}^{[l-pi+1]}u^{pi-1} dz$. Then this map induces an isomorphism 
            $$
            \widetilde{\ker}(D_x'dx\oplus D_z'dz)_l/\ima(D_z'dz,D_x'dx)_{l+1}\simeq \oplus_{1\leq i\leq \lfloor\frac{l}{p}\rfloor}(\F_p[x]/(x^{p+1}))xt_{zx}^{[l-pi+1]}u^{pi-1} dz.
            $$

            \item[(3)] Suppose that $p\mid l+1$, and consider the $R$-linear map 
            $$
            \widetilde{\ker}(D_x'dx\oplus D_z'dz)_l\rightarrow (R/(l+1,x^{p+1}))u^ldz\oplus(\oplus_{1\leq i\leq \lfloor\frac{l}{p}\rfloor}(\F_p[x]/(x^{p+1}))xt_{zx}^{[l-pi+1]}u^{pi-1} dz),
            $$
            which sends an element $\tilalpha=\tilf_lu^ldz+(\sum_{\text{$0\leq l'<l$, $p\nmid l'+1$}}\tilb_{l'}\beta_{l,l'
            })+(\sum_{i=1}^{\lfloor\frac{l}{p}\rfloor}\tilg_{i}p\gamma_{l,i}+x\tilg_i't_{zx}^{[l-pi+1]}dzu^{pi-1} )+\tilg_0p\gamma_{l,0}$ to 
            \begin{align*}
            &(\tilf_l+(\sum_{j=1}^{p-1}\frac{(-x)^{j}}{l(l-1)\cdots(l-j+1)}\tilb_{l-j})-\frac{(-x)^{p}}{l(l-1)\cdots(l-p+2)}\tilg_{\lfloor\frac{l}{p}\rfloor})u^ldz+\\
            &\sum_{i=1}^{\lfloor\frac{l}{p}\rfloor}(\tilg_i'-(\sum_{j=1}^{p-1}\frac{(-x)^{j-1}}{(pi-1)(pi-2)\cdots(pi-j)}\tilb_{pi-j-1})-\frac{(-x)^{p}}{(pi-1)\cdots(pi-p+1)}\tilg_{i-1})xt_{zx}^{[l-pi+1]}u^{pi-1} dz.
            \end{align*}
            Then this map induces an isomorphism 
            \begin{align*}
                \ &\widetilde{\ker}(D_x'dx\oplus D_z'dz)_l/\ima(D_z'dz,D_x'dx)_{l+1}\\
                & \simeq (R/(l+1,x^{p+1}))u^ldz\oplus(\oplus_{1\leq i\leq \lfloor\frac{l}{p}\rfloor}(\F_p[x]/(x^{p+1}))xt_{zx}^{[l-pi+1]}u^{pi-1} dz).
            \end{align*}
        \end{itemize}
    \end{claim}
    Indeed, Claim \ref{claim:H1} follows immediately from Claim \ref{claim:im1} in Computation \ref{computation:im1}. 
\end{computation}

We summarize the previous computations as follows. 

\begin{thm}\label{thm:computation}
    Let $R=\Z_p[x]/(px)$. Recall from Notation \ref{notation:Hopf-algebroid} the (graded) Hopf algebroid
    $$
    (A,\Gamma)=(\pi_*\THH(R/\Sphere_p[z,x]),\pi_*\THH(R/\Sphere_p[z_0,x_0,z_1,x_1])).
    $$ 
    Also recall the descent spectral sequence in Construction \ref{construction:descent-spectral-sequence}
    $$
    E^1_{i,j}=\pi_j \THH(R/\Sphere_p[z,x]^{\otimes_{\Sphere_p}(-i+1)})^\wedge_p\Rightarrow \pi_{i+j}\THH(R)^\wedge_p,
    $$
    and by Lemma \ref{lem:Hopf-algebroid} (4), we have 
    $$
    E^2_{i,j}\simeq \Ext^{-i,j}_{\Gamma}(A,A). 
    $$
    Moreover, by Proposition \ref{prop:resolution} we know that the terms $E^2_{i,j}$ are computed by the complex
    $$
    R\langle t_{zx}\rangle[u]\xrightarrow{(D'_zdz,D'_x dx)} R\langle t_{zx}\rangle[u]dz\oplus R\langle t_{zx}\rangle[u]dx\xrightarrow{D'_xdx\oplus D'_zdz} R\langle t_{zx}\rangle[u](dz\wedge dx),
    $$
    where $|u|=2$, $t_{zx}^{[k]}=2k$, and $|dz|=|dx|=2$, $|dz\wedge dx|=4$, and $D'_z: R\langle t_{zx}\rangle[u]\rightarrow  R\langle t_{zx}\rangle[u]$ is the $R$-linear map sending $t_{zx}^{[k]}u^n $ to $nt_{zx}^{[k]}u^{n-1} +xt_{zx}^{[k-1]}u^n $, and $D'_x: R\langle t_{zx}\rangle[u]\rightarrow  R\langle t_{zx}\rangle[u]$ is the $R[u]$-linear map sending $t_{zx}^{[k]}$ to $pt_{zx}^{[k-1]}$. By Corollary \ref{cor:degenerate}, we know that $E^2_{i,j}=0$ if $i\neq 0,-1,-2$ or $j$ is odd, and therefore this spectral sequence collapses at the $E^2$-page. 
    
    Now, for $l\geq 1$, we set 
    \begin{align*}
        \epsilon_{l} = \sum_{n+k=l,p\lfloor\frac{l}{p}\rfloor\leq n\leq l}(-1)^{n-p\lfloor\frac{l}{p}\rfloor}\frac{x^{n-p\lfloor\frac{l}{p}\rfloor}}{n(n-1)\cdots(p\lfloor\frac{l}{p}\rfloor+1)} t_{zx}^{[k]}u^n ;
    \end{align*}
    and for $0\leq l'< l$ such that $l'+1$ is not divisible by $p$, we set 
    $$
    \beta_{l,l'}=t_{zx}^{[l-l']}u^{l'} dz+\sum_{n=l'+1}^{\min\{p(\lfloor \frac{l'}{p}\rfloor+1)-1,l\}}\frac{(-x)^{n-l'}p}{n(n-1)\cdots(l'+1)}t_{zx}^{[l-n]}u^n dx;
    $$
    also, for $1\leq i\leq \lfloor\frac{l}{p}\rfloor$, we set 
    $$
    \gamma_{l,i}=it_{zx}^{[l-pi+1]}u^{pi-1} dz+\sum_{n=pi}^{\min\{p(i+1)-1,l\}}\frac{(-x)^{n-pi}}{n(n-1)\cdots(pi+1)}t_{zx}^{[l-n]}u^n dx,
    $$
    and finally,
    $$
    \gamma_{l,0}=\sum_{n=0}^{\min\{p-1,l\}}\frac{(-x)^{n-pi}}{n(n-1)\cdots(pi+1)}t_{zx}^{[l-n]}u^n dx.
    $$
    Then the following statements hold:
    \begin{itemize}
        \item[(1)] We have $E^2_{0,0}\simeq R$. For $l\geq 1$, we have 
        $$
        E^2_{0,2l}\simeq \F_p[x]x\epsilon_l.
        $$

        \item[(2)] We have $E^2_{-1,0}\simeq 0$, $E^2_{-1,2}\simeq \F_p[x]dx$. Let $l\geq 1$ be an integer. Then the following statements hold:
        \begin{itemize}
            \item[(i)] If $p\nmid l+1$, then $E^2_{-1,2l+2}$ fits into a non-split exact sequence 
            $$
            0\rightarrow \oplus_{1\leq i\leq \lfloor\frac{l}{p}\rfloor}(\F_p[x]/(x^{p+1}))xt_{zx}^{[l-pi+1]}u^{pi-1} dz\rightarrow E^2_{-1,2l}\rightarrow  \oplus_{0\leq i\leq \lfloor\frac{l}{p}\rfloor}\F_p[x]\gamma_{l,i}\rightarrow 0.
            $$

            \item[(ii)] If $p\mid l+1$, then $E^2_{-1,2l+2}$ fits into a non-split exact sequence 
            $$
            0\rightarrow (R/(l+1,x^{p+1}))u^ldz\oplus(\oplus_{1\leq i\leq \lfloor\frac{l}{p}\rfloor}(\F_p[x]/(x^{p+1}))xt_{zx}^{[l-pi+1]}u^{pi-1} dz)\rightarrow E^2_{-1,2l}\rightarrow \oplus_{0\leq i\leq \lfloor\frac{l}{p}\rfloor}\F_p[x]\gamma_{l,i}\rightarrow 0.
            $$
        \end{itemize}

        \item[(3)] We have $E^2_{-2,0}=E^2_{-2,2}=E^2_{-2,4}=0$. Let $l\geq 1$ be an integer, then 
        $$
        E^2_{-2,2l+4}\simeq \oplus_{1\leq i\leq \lfloor\frac{l+1}{p}\rfloor}(\F_p[x]/(x^p))t_{zx}^{l-pi+1}u^{pi-1} (dz\wedge dx).
        $$
    \end{itemize}
\end{thm}

\begin{proof}
    The statement (1) summarizes Computation \ref{computation:H0}. The statement (2) summarizes Computation \ref{computation:ker1}, Computation \ref{computation:im1}, and Computation \ref{computation:H1}. The statement (3) summarizes Computation \ref{computation:H2}.
\end{proof}

%\begin{cor}\label{cor:computation}
    
%\end{cor}

\begin{rem}\label{rem:extension-problem-and-multiplicative-structure}
    The extension problem for the spectral sequence in Theorem \ref{thm:intro-computation} remains unsolved, and the multiplicative structure has not been specified. However, in a forthcoming work with Wei Yang \cite{GY26}, we will resolve the extension problem and specify the multiplicative structure by choosing a suitable differential graded algebra model of $\THH(\Z_p[x]/(px))^\wedge_p$. 
\end{rem}

\end{document}